\documentclass[11pt, reqno]{amsart}
\usepackage[english]{babel}
\usepackage[table]{xcolor}
\usepackage[utf8]{inputenc}
\usepackage{amsmath, amsfonts, amsthm, amscd, amssymb, fullpage, fancyhdr, upgreek, amssymb, hyperref, enumerate, comment, siunitx, enumitem, graphicx, mathrsfs, mathtools, xcolor,
microtype}
\usepackage{listings}
\usepackage[scr]{rsfso}
\usepackage{amssymb}

\newcolumntype{P}[1]{>{\centering\arraybackslash}p{#1}}

\usepackage[parfill]{parskip}

\newtheorem{theorem}{Theorem}[section]
\newtheorem{proposition}[theorem]{Proposition}
\newtheorem{lemma}[theorem]{Lemma}
\newtheorem{definition}[theorem]{Definition}

\newtheorem*{remark}{Remark}
\newtheorem*{notation}{Notation}

\numberwithin{equation}{section}

\newcommand{\bburl}[1]{\textcolor{blue}{\url{#1}}}

\usepackage{tikz}
\usepackage{tkz-tab}
\usepackage{tkz-graph}
\usetikzlibrary{shapes.geometric,positioning}

\newcommand{\hr}[1]{\href{#1}{\url{#1}}}

\newcommand{\nocontentsline}[3]{}

\newcommand{\nth}{^{\text{th}}}
\newcommand{\z}{\mathbb{Z}}
\newcommand{\q}{\mathbb{Q}}
\newcommand\rank{\operatorname{rank}}
\newcommand\ee{\operatorname{e}}

\title{Counting solutions to quadratic forms in eight variables of off-diagonal rank $3$}
\author{Aleksandra Kowalska}
\thanks{Mathematical Institute, University of Oxford}

\begin{document}

\begin{abstract}
In this note we count the number of solutions to a non-degenerate quadratic form in eight prime variables of off-diagonal rank $3$. It is a continuation of work by L. Zhao, who counted solutions to forms in at least nine variables, by B. Green, who counted solutions to `generic' (which implies off-diagonal rank $4$) forms in eight variables, and by J. Dobrowolski, who counted solutions to forms in eight variables of off-diagonal rank at most $2$.
\end{abstract}

\maketitle

\section{Introduction}
For a non-degenerate, homogeneous quadratic form $Q$ in $n \geqslant 5$ prime variables and $X, N \in \mathbb{N}$, $L=\log X$, we expect to have (for $K$ arbitrarily large):
\begin{equation}
\label{thesis}
    \sum_{x \in [X]^n}\Lambda(x)\cdot1_{Q(x)=N}=\mathfrak{S}_Q(N)\mathfrak{J}_Q(N, X) + O_K(X^{n-2}L^{-K}),
\end{equation}
where:
\begin{equation}
\label{s_def}
    \mathfrak{S}_Q(N):=\sum_{q > 1}\sum_{a \in (\z/q\z)^{\times}} e \Big( -\frac{aN}{q} \Big) \sum_{x_1, \dots, x_n \in (\z/q\z)^{\times}} e \Big( \frac{aQ(x_1, \dots, x_n)}{q} \Big)
\end{equation}
and
\begin{equation}
\label{j_def}
    \mathfrak{J}_Q(N, X):=\int_{t=-\infty}^{+\infty}e(-tN)\int_{x \in [X]^n}e(tQ(x)) \ dt \ dx_1 \dots dx_8.
\end{equation}
We note that $\mathfrak{S}_Q(N)=\Theta(1)$ and $\mathfrak{S}_Q(N)=\Theta(X^{n-2})$, as shown in \cite{nine}.

\begin{remark}
    We will often think about $Q$ as its symmetric matrix, i.e. $Q(x)=x^TQx$ (abusing the notation and denoting $Q$'s matrix also by $Q$).
\end{remark}

First, the above was proven by Hua for $Q$ a diagonal form in at least 5 variables \cite{diagonal}. Then, Liu handled a wide class of forms in at least 10 variables \cite{liu}, and Zhao proved the above result for forms in at least 9 variables \cite{nine}. A few years ago Green improved this, proving the formula for `generic' forms in 8 variables \cite{eight}, where `generic' forms are defined as follows:

A quadratic form $Q$ in eight variables in generic if, for some permutation of variables, we may write $Q=\begin{pmatrix}
    A & B\\ B^T & C
\end{pmatrix}$, where $A, B, C \in M_{4 \times 4}$ such that:
\begin{enumerate}
    \item $A, B, C$ all have rank $4$,
    \item $B^{-1}AB^{-T}C$ has four distinct eigenvalues in $\bar{\mathbb{Q}}$
\end{enumerate}
(the definition used by Green is slightly different, but equivalent to this one).

For a quadratic form $Q$ represented as a matrix, we define its off-diagonal rank to be the maximal $k$ such that $Q$ has a $k \times k$ invertible submatrix, which does not include its diagonal entries.

We notice that if $Q$ is an $n \times n$ matrix, its off-diagonal rank is at most $\lfloor \frac{n}{2} \rfloor$, so in particular an $8 \times 8$ matrix can have an off-diagonal rank at most $4$.

We notice (after Dobrowolski) that $Q$ can be written in a form $\begin{pmatrix}
    A & B\\ B^T & C
\end{pmatrix}$ with $B$ invertible if and only if its off-diagonal rank is $4$. Hence, if a form in 8 variables is generic, its off-diagonal rank equals 4.

A form is of off-diagonal rank $0$ if and only if it is diagonal, so this case is already included in Hua's result. In \cite{small_off_diag_rank}, Dobrowolski counted the number of solutions to quadratic forms in eight variables of off-diagonal rank 1 or 2. In this note, I prove the following result:

\begin{theorem}
\label{the_main_theorem}
    For $Q$ a non-degenerate and homogeneous quadratic form in $8$ variables with off-diagonal rank $3$, the number of solutions to $Q(x_1, \dots, x_8)=N$ with $x_1, \dots, x_8 \leqslant X$ is
    \begin{equation}
        \mathfrak{S}_Q(N)\mathfrak{J}_Q(N, X) + O_K(X^{6}L^{-K})
    \end{equation}
    for $K$ an arbitrary large integer (where $L = \log X$).
\end{theorem}

Similarly to Zhao and Dobrowolski, this is done using the circle method.

\section{Preliminaries}

\begin{definition}
    For any $X, B \in \mathbb{Z}_+$, let us define: $$\mathfrak{M}(B):=\bigcup_{1 \leqslant b \leqslant B} \bigcup_{\substack{1 \leqslant a \leqslant b, \\ \gcd(a, b)=1}} \left\{ \alpha: \Big| \alpha - \frac{a}{b} \Big| \leqslant \frac{B}{bX^2} \right\}, \ \mathfrak{m}(B)=\Big[\frac{1}{X},1+\frac{1}{X}\Big] \setminus \mathfrak{M}(B).$$
    We usually take $B=L^{5K}$ and denote $\mathfrak{M}(L^{5K}), \ \mathfrak{m}(L^{5K})$ simply by $\mathfrak{M}, \ \mathfrak{m}$.
\end{definition}

The above definition of minor and major arcs is the same as in \cite{nine} (with a slightly different notation). The only place where this definition is used is below when quoting Zhao's results and when quoting Lemma \ref{first_bounding_lemma}.

In \cite{nine}, Zhao proved that a homogeneous quadratic equation $Q$ in $n \geqslant 9$ variables has
\begin{equation*}
    \mathfrak{S}_Q(N)\mathfrak{J}_Q(N, X) + O(X^{6}L^{-K})
\end{equation*}
solutions to $Q(x_1, \dots, x_n)=N$ in $x_i \leqslant X$ prime. He did so using the circle method, in the following two propositions:

\begin{proposition}
\label{zhao_major}
    For a quadratic form in $n \geqslant 5$ variables and $\alpha \in \mathfrak{M}(L^{5K})$, we have: $$\int_{\alpha \in \mathfrak{M}}\sum_{x \in [X]^n} \Lambda(x) e(\alpha(Q(x)-N))=\mathfrak{S}_Q(N)\mathfrak{J}_Q(N, X) + O(X^{6}L^{-K/4}).$$
\end{proposition}

\begin{proposition}
    For a quadratic form in $n \geqslant 9$ variables and any $\alpha \in \mathfrak{m}(L^{5K})$, we have: $$\int_{\alpha \in \mathfrak{m}} \sum_{x \in [X]^n} \Lambda(x) e(\alpha(Q(x)-N))=O(X^{6}L^{-K/{20}}).$$
\end{proposition}

Since Proposition \ref{zhao_major} works for $n \geqslant 5$, to prove Theorem \ref{the_main_theorem} we only need to bound the integral over the minor arcs.

Most of the proof (especially the cases with higher ranks) are very similar to the proof techniques from \cite{nine} and \cite{small_off_diag_rank}. The main ideas appearing here that do not appear in either of these papers are the following:

\begin{itemize}
    \item Abstracting out a method with Cauchy-Schwarz used in \cite{nine} and \cite{small_off_diag_rank}. This is the main contribution and it allowed me to handle more complex variable substitution.
    \item Applying the lemmas from Section \ref{section_structure} on the possible form of a matrix with off-diagonal rank $3$ multiple times to the same matrix to get more detailed information on its form (in Section \ref{section_twos}).
\end{itemize}

\begin{notation}
    We denote by $q_{ij}$ the entry of $Q$ in its $i\nth$ row and $j\nth$ column (sometimes, if the expressions for $i$ or $j$ are a bit more comples, we denote this entry as $q_{i; j}$).
    
    We also denote by $Q_{R; \ C}$ the $|R| \times |C|$ submatrix of $Q$ consisting of entries $q_{rc}$ with $r \in R$ and $c \in C$. Finally, in indices we sometimes write $i:j$ for $\{i, i+1, i+2, \dots, j\}$ and $[i]$ for $\{1, 2, \dots, i\}$.

    Let us also write $x_1^n \ll X$ if $x_i \leqslant c_iX$ for some constants $c_i$.
\end{notation}

In Section \ref{abstracting_section}, I present the main lemma used for proving Theorem \ref{the_main_theorem}. Section \ref{section_special_cases} applies that lemma to prove Theorem \ref{the_main_theorem} in three cases of matrices failing to have off-diagonal rank greater than $3$ in simple ways. Then, in Section \ref{section_structure}, I describe (after Zhao \cite{nine}) possible forms of $Q$ with an off-diagonal rank $3$. The following three sections, \ref{section_three_three}, \ref{section_three_two} and \ref{section_twos} prove Theorem \ref{the_main_theorem} in three propositions, depending on the form of $Q$. The final, brief section puts together these propositions to prove the theorem.

\textbf{Acknowledgements.} I would like to thank my supervisor Ben Green for suggesting the project and helpful discussions.

\section{Abstracting the Zhao and Dobrowolski's technique}
\label{abstracting_section}

In this section, we introduce a lemma which packages the usage of Cauchy-Schwarz for bounding minor arcs and is the main tool used in proving Theorem \ref{the_main_theorem}.

First, I state the lemma and explain how it is usually used, then prove it, indicating where its fours conditions come from. The next section contains three simple examples of applying the lemma to prove Theorem \ref{the_main_theorem} for $Q$s in specific forms.

\begin{lemma}
\label{the_main_tool_lemma}
    Let us consider an expression: \begin{multline*}
        S(\alpha) := \sum_{y_1, \dots, y_{n+m} \ll X} e(\alpha \cdot y^T M y) \ \Lambda(y_1) \dots \Lambda(y_n) \prod_{j=1}^m 1_{y_{n+j}=\sum_{i=1}^n y_il_{ji}} \\ = \sum_{y_1, \dots, y_{n+m} \ll X} e(\alpha \cdot y^TMy) \Lambda(y_1) \dots \Lambda(y_n) \prod_{j=1}^{m} \int_{\beta_j \in [0, 1]}e\Big(\beta_j\big(y_j - L_j \cdot \begin{pmatrix} y_{1} & \dots & y_{n} \end{pmatrix}^T \big) \Big) \ d\beta_j,
    \end{multline*}
    where $M$ is an $(n+m) \times (n+m)$ matrix, $y=\begin{pmatrix}
        y_1 & \dots & y_{n+m}
    \end{pmatrix}^T$ and $L_j = \begin{pmatrix} l_{j1} & l_{j2} & \dots & l_{jn}\end{pmatrix}$.

    This lemma presents sufficient conditions for us to have (for an arbitrarily large $K$): $$\big| \int_{\alpha \in \mathfrak{m}}S(\alpha)e(-N\alpha) \ d\alpha \big|=O_K(X^{n-2}L^{-K}).$$

    Suppose that there exist disjoint tuples $A=(a_1, \dots, a_u), B=(b_1, \dots, b_v), C=(c)$, such that each element of $[n+m]$ appears in exactly one of them, exactly once (so $u+v+1=n+m$).

    Now we need to introduce some notation: let $\mathcal{A}$ be an $m \times u$ matrix, whose $i\nth$ column is either $\begin{pmatrix}l_{1a_i} & \dots & l_{ma_i} \end{pmatrix}^T$ (if $a_i \leqslant n$) or has $0$s everywhere, except for $(-1)$ on its $(a_i-n)\nth$ position (if $a_i>n$).

    Finally, let us denote by $\mathcal{A}_1$ a matrix consisting of the first $u-m$ columns of $\mathcal{A}$ and by $\mathcal{A}_2$ a matrix consisting of its last $m$ columns. Let $\mathcal{B}, \mathcal{B}_1, \mathcal{B}_2$ be defined analogously.

    We need $A, B, C$ to satisfy the following conditions:
    
    $(1):$ Denoting $z=\begin{pmatrix}y_{a_1} & \dots & y_{a_u}\end{pmatrix}^T$ and $t=\begin{pmatrix}y_{b_1} & \dots & y_{b_u}\end{pmatrix}^T$, there exist symmetric matrices $P_A, P_B$ and a constant $d$ such that $$y^TMy = z^TP_Az + t^TP_Bt + dy_c^2.$$
    $(2):$ We have $d \neq 0.$
    
    $(3A):$ $\mathcal{A}_2$ is invertible, $(3B):$ $\mathcal{B}_2$ is invertible (we note that this condition implies $u, v \geqslant m$).

    $(4A):$ The rank of $P_A \cdot\begin{pmatrix}D_{u-m} \\ -\mathcal{A}_2^{-1} \cdot \mathcal{A}_1\end{pmatrix}$ is at least $2$, $(4B):$ the rank of $P_B \cdot\begin{pmatrix}D_{v-m} \\ -\mathcal{B}_2^{-1} \cdot \mathcal{B}_1\end{pmatrix}$ is at least $2$, where $\begin{pmatrix}D_{u-m} \\ -\mathcal{A}_2^{-1} \cdot \mathcal{A}_1\end{pmatrix}$ is a $u \times (u-m)$ matrix whose first $u-m$ rows form the identity matrix and whose next $m$ rows form the matrix $-\mathcal{A}_2^{-1} \cdot \mathcal{A}_1$.
\end{lemma}

When we apply the lemma, the variables $y_1, \dots, y_n$ are the original variables $x_1, \dots, x_n$ from the $Q(x_1, \dots, x_n)=N$ equation and the variables $y_{n+1}, \dots, y_{n+m}$, $y_{n+j}=\sum_{i=1}^n y_il_{ji}$ are auxiliary variables, introduced to simplify the expression for $Q$ (more specifically, to be able to write it as a sum of three expressions with different variables, $Q=P_A+P_B+P_C$).

The lemma says that if one of the variables has no relations to the other ones, and the remaining variables can be separated into two `independent' groups in an appropriate way, we get the required bound on the minor arc of the integral counting the number of solutions. It `packages' four lemmas that are used multiple times together in \cite{nine} and \cite{small_off_diag_rank}. Let us now cite Lemmas \cite[4.3, 5.6, 5.7, 5.8]{nine} and use them to prove Lemma \ref{the_main_tool_lemma}.

\begin{lemma}
\label{first_bounding_lemma}
    For any $\alpha \in \mathfrak{m}, \ \beta \in \mathbb{R}, \ d \in \mathbb{Q} \setminus \{0\}$, we have:
    $$\sum_{x \leqslant X} \Lambda(x)e(\alpha dx^2 + x\beta) \ll_{d, K} XL^{-K}.$$
\end{lemma}

\begin{lemma}
\label{second_bounding_lemma}
    For $C$ an $n \times n$ symmetric matrix, $H$ an $n \times m$ matrix and any positive real weight function $w$, we have:
    $$\int_{\substack{\alpha \in [0,1], \\ \beta \in [0,1]^m}}\Big| \sum_{x \ll X} w(x)e(\alpha x^T C x + x^T H \beta) \Big|^2 \ d\alpha \ d\beta \ll \sum_{\substack{x, y \ll X, \\ x^TCx=y^TCy, \\ Hx = Hy}}w(x)w(y).$$
\end{lemma}

\begin{lemma}
\label{third_bounding_lemma}
    For $C$ an $n \times n$ symmetric matrix and $H$ an $n \times m$ matrix, we have:
    $$\# \left\{ |x|, |y| \ll X : x^TCx=y^TCy, \ Hx = Hy \right\} \ll \# \left\{ |x|, |y| \ll X : x^TCy = 0, \ Hx = 0 \right\}.$$
\end{lemma}

\begin{lemma}
\label{fourth_bounding_lemma}
    For $B$ an $n \times m$ matrix with $\rank(C) \geqslant 2$, we have:
    $$\# \left\{ |x|, |y| \ll X : x^TCy = 0 \right\} \ll X^{n+m-2}L.$$
\end{lemma}

\begin{proof}[Proof of Lemma \ref{the_main_tool_lemma}]
    To simplify notation in the proof, let us define: $\Lambda_{i \leqslant n}(x)=\begin{cases}
        \Lambda(x) & \text{if } i \leqslant n \\
        1 & \text{otherwise}
    \end{cases}$.

    First, using assumption $(1)$ and applying Cauchy-Schwarz, we can bound:
    \begin{equation}
        \Big| \int_{\alpha \in \mathfrak{m}} S(\alpha)e(-N\alpha) d\alpha \Big| \leqslant \int_{\alpha \in \mathfrak{m}} |S(\alpha)| \ d\alpha \leqslant T_1 \cdot T_2^{1/2} \cdot T_3^{1/2},
    \end{equation}
    where:
    \begin{equation*}
        T_1 = \sup_{\substack{\alpha \in \mathfrak{m}, \\ \gamma \in \mathbb{R}}} \Big| \sum_{y_c \ll X} \Lambda(y_c) e(\alpha d y_c^2 + \gamma y_c) \Big|,
    \end{equation*}
    \begin{equation*}
        T_2 =\int_{\substack{\alpha \in \mathfrak{m}, \\ \beta \in [0, 1]^m}} \bigg| \sum_{y_{a_1}, \dots, y_{a_u} \ll X} \ \prod_{i=1}^n \Lambda_{a_i \leqslant n}(y_{a_i}) \cdot e\Big(\alpha y^T P_A y + \sum_{i=n+1}^{n+m} \beta_i\big(1_{i \in A}y_{i} - \sum_{j=1}^u L_{ia_j} y_{a_j}\big)\Big) \bigg|^2 d \alpha \ d \beta
    \end{equation*}
    and
    \begin{equation*}
        T_3 = \int_{\substack{\alpha \in \mathfrak{m}, \\ \beta \in [0, 1]^m}} \bigg| \sum_{y_{b_1}, \dots, y_{b_v} \ll X} \ \prod_{i=1}^n \Lambda_{b_i \leqslant n}(y_{b_i}) \cdot e\Big(\alpha y^T P_B y + \sum_{i=n+1}^{n+m} \beta_i\big(1_{i \in B}y_{i} - \sum_{j=1}^v L_{ib_j} y_{b_j}\big)\Big) \bigg|^2 d \alpha \ d \beta.
    \end{equation*}
    The expression $T_1$ can be bounded by $O(XL^{-K})$, using Lemma \ref{first_bounding_lemma} and assumption $(2)$.
    
    As $n=u+v+1-m$, it suffices to bound $T_2$ by $O(X^{2u-m-2}L^{C})$ and $T_3$ by $O(X^{2v-m-2}L^C)$ for some constant $C$. Let us now bound $T_2$ (bounding $T_3$ is completely analogous).

    By Lemmas \ref{second_bounding_lemma} and \ref{third_bounding_lemma}, we have $$T_2 \ll L^{2n} \# \{ |y|, |z| \ll X: z^TP_A y = 0, \mathcal{A}z = 0 \}$$
    for $y, z$ vectors of lengths $u$.
    We note that the expression above has $2u$ variables. Using condition $(3)$ that $\rank(\mathcal{A})=m$, we can write $m$ of variables in the $z$ vector in terms of the remaining $u-m$ variables. After doing this, we have (for $\mathcal{A}$ etc. defined as in the condition $(4)$): $$T_2 \ll L^{2n} \# \{ |y_1^u|, |z_1^{u-m}| \ll X : y^T P_A \begin{pmatrix}D_{u-m} \\ -\mathcal{A}_2^{-1} \cdot \mathcal{A}_1\end{pmatrix} z = 0 \}.$$ Condition $(4)$ ensures that we can apply Lemma \ref{fourth_bounding_lemma} to bound $T_2 \ll L^{2n+1}X^{2u-m-2}$, as required.
\end{proof}

\begin{remark}
    We note that the lemma essentially allows us to forget that what we are doing is bounding expressions using Cauchy-Schwarz, and instead just think about the quadratic in question in terms of linear algebra.
\end{remark}

\begin{remark}
    We note that when $Q(x_1, \dots, x_8)=R(x_1, \dots, x_i) + S(x_{i+1}, \dots, x_7) + tx_8^2$ for $2 \leqslant i \leqslant 5$, we may apply Lemma \ref{the_main_tool_lemma} without introducing any extra variables. We would have: $A=(x_1, \dots, x_i), \ B=(x_{i+1}, \dots, x_7), \ C=(x_8)$, $P_A=R, \ P_B=S$ and $c=t \neq 0$. The matrices $\mathcal{A}, \mathcal{B}$ are trivial with $0$ rows, so the third and fourth conditions are trivially satisfied (observing that $\rank(Q)=8$ implies $\rank(R)=i \geqslant 2$ and $\rank(S)=7-i \geqslant 2$).
\end{remark}

\section{The proof in three special cases.}
\label{section_special_cases}
This section contains three examples of applying Lemma \ref{five_five_zero_lemma} in three `extreme' cases of matrices with off-diagonal rank $3$. They will all be used later in the following sections as components of the proof of Theorem \ref{the_main_theorem} in the general case, but are presented here separately both as simpler examples of applying Lemma \ref{the_main_tool_lemma} and for greater clarity of the proof later.

\begin{lemma}
\label{five_five_zero_lemma}
    Let $Q$ be a $8 \times 8$ matrix such that there exist five pairwise different indices $i_1, \dots,  i_5$ with $q_{jk}=0$ for any $j, k \in \{i_1, \dots, i_5\}, \ j \neq k$. Then Theorem \ref{the_main_theorem} holds for $Q$.
\end{lemma}

\begin{proof}
    Without loss of generality $i_1, \dots, i_5 = 1, \dots, 5$.
    
    In the proof, we will assume that $\rank(Q_{1:5; \ 6:8})=3$. The proofs in the other cases are analogous to this one but simpler (in each case, we will need to introduce $\rank(Q_{1:5; \ 6:8})$ new variables), and so they are left.

    Let $y=\begin{pmatrix}x_1 & x_2 & x_3 & x_4 & x_5\end{pmatrix}^T$ and $z=\begin{pmatrix}x_6 & x_7 & x_8\end{pmatrix}^T$ (so $z_i=x_{5+i}$).
    
    Let us introduce three helping variables:
    \begin{equation}
        v_i=2\sum_{j=1}^5q_{ij}y_j+\sum_{j=1}^3q_{i;5+j}z_j, \ \ 1 \leqslant i \leqslant 3
    \end{equation}
    We need to bound by $O(X^6L^{-K})$ the following:
    \begin{equation*}
        \int_{\substack{\alpha \in \mathfrak{m}, \\ \beta_1^3 \in [0, 1]}} \sum_{\substack{y_1^5 \in [X]^5, \\ z_1^3 \in [X]^3, \\ v_1^3 \ll X}} \ee\bigg( \alpha \Big( y^TDy+\sum_{i=1}^3 v_i z_i - N \Big)
        +\sum_{i=1}^3 \beta_i \Big( v_i-2\sum_{j=1}^5q_{ij}y_j-\sum_{j=1}^3q_{i;5+j}z_j \Big) \bigg) d\alpha \ d\beta_i
    \end{equation*}
    (where $D$ is a diagonal matrix).

    Before dividing the variables into tuples and applying Lemma \ref{the_main_tool_lemma}, we show that there exists a permutation $(i_1 \dots i_5)$ of $(1 \dots 5)$ such that:
    \begin{itemize}
        \item $q_{i_18} \neq 0$,
        \item $\begin{pmatrix}
            q_{6;i_2} & q_{6;i_3}\\
            q_{7;i_2} & q_{7;i_3}
        \end{pmatrix}$ is invertible,
        \item $d_{i_4}, d_{i_5} \neq 0$.
    \end{itemize}

    We claim that there exists $I \subset \{1, 2, 3, 4, 5\}, \ |I|=3$ such that $\rank(Q_{I; \ 6, 7, 8})=3$ and $d_l \neq 0$ for $l \in \{1, \dots, 5\} \setminus I$. Indeed, let $J=\{ 1 \leqslant j \leqslant 5: d_j=0\}$. We note that $|J| \leqslant 3$ (as $\rank(D)+3+3 \geqslant 8$, so $\rank(D) \geqslant 2$) and that $Q_{J; \ 6,7,8}$ has linearly independent rows, as $Q$ is non-degenerate. Hence, since $\rank(Q_{1:5; \ 6:8})=3$ (from the assumption stated at the beginning), we may pick the lacking $3-|J|$ indices to have $J \subset I$ and $\rank(Q_{I; \ 6, 7, 8})=3$.

    Without loss of generality $I=\{1,2,3\}$. Let us denote $M:=Q_{1,2,3; \ 6,7,8}$.

    Since $M$ is invertible, from a formula for its determinant there exists $k \in \{1, 2, 3\}$ such that $q_{k8}=M_{k3} \neq 0$ and $\rank(M_{\{1, 2, 3\} \setminus \{k\}; \ 1,2})=2$. We may assume $k=1$. Then, we take $i_j=j$ for $1 \leqslant j \leqslant 5$, which satisfy the requirements.

    If another ordering of $1, \dots, 5$ satisfies the requirements, we can reorder the variables to have $i_j=j$.

    Now we apply Lemma \ref{the_main_tool_lemma} with the following tuples:
    $$A=(z_{1}, z_{2}, y_{1}, v_{1}, v_{2} ), \ B = (z_{3}, y_{4}, y_{2}, y_{3}, v_{3}), \ C=(y_{5}).$$ We need to show that they satisfy the conditions of the lemma.

    $(1):$ This condition is clearly satisfied with  $$P_A=\begin{pmatrix}
        0 & 0 & 0 & 1/2 & 0\\
        0 & 0 & 0 & 0 & 1/2\\
        0 & 0 & d_{1} & 0 & 0\\
        1/2 & 0 & 0 & 0 & 0\\
        0 & 1/2 & 0 & 0 & 0\\
    \end{pmatrix}, \ P_B=\begin{pmatrix}
        0 & 0 & 0 & 0 & 1/2\\
        0 & d_{ 4} & 0 & 0 & 0\\
        0 & 0 & d_{ 2} & 0 & 0\\
        0 & 0 & 0 & d_{ 3} & 0\\
        1/2 & 0 & 0 & 0 & 0
    \end{pmatrix}, \ c=d_{ 5}.$$
    $(2):$ As we showed just before defining the tuples, we have $d_{ 5} \neq 0.$

    $(3A):$ We have $\mathcal{A}=\begin{pmatrix}
        q_{66} & q_{67} & 2q_{61} & -1 & 0\\
        q_{76} & q_{77} & 2q_{71} & 0 & -1\\
        q_{86} & q_{87} & 2q_{81} & 0 & 0
    \end{pmatrix}$, so $\mathcal{A}_1=\begin{pmatrix}
        q_{66} & q_{67}\\
        q_{76} & q_{77}\\
        q_{86} & q_{87}
    \end{pmatrix}$, $\mathcal{A}_2=\begin{pmatrix}
        2q_{61} & -1 & 0\\
        2q_{71} & 0 & -1\\
        2q_{81} & 0 & 0
    \end{pmatrix}$ and hence $\mathcal{A}_2$ is invertible as $q_{81} \neq 0$.

    $(3B):$ We have $\mathcal{B}=\begin{pmatrix}
        q_{68} & 2q_{64} & 2q_{6 2} & 2q_{6 3} & 0 \\
        q_{78} & 2q_{74} & 2q_{7 2} & 2q_{7 3} & 0 \\
        q_{88} & 2q_{84} & 2q_{8 2} & 2q_{8 3} & -1
    \end{pmatrix}$, so $\mathcal{B}_1=\begin{pmatrix}
        q_{68} & 2q_{64} \\
        q_{78} & 2q_{74} \\
        q_{88} & 2q_{84}
    \end{pmatrix}$, $\mathcal{B}_2=\begin{pmatrix}
        2q_{6 2} & 2q_{6 3} & 0 \\
        2q_{7 2} & 2q_{7 3} & 0 \\
        2q_{8 2} & 2q_{8 3} & -1
    \end{pmatrix}$ and so $\mathcal{B}_2$ is invertible as (from our assumptions) $\begin{pmatrix}
        q_{6 2} & q_{6 3}\\
        q_{7 2} & q_{7 3}
    \end{pmatrix}$ is.

    $(4A):$ We need the rank of $P_A \cdot \begin{pmatrix}
        \begin{matrix}
            1 & 0\\
            0 & 1
        \end{matrix}\\
        -\mathcal{A}_2^{-1} \cdot \mathcal{A}_1
    \end{pmatrix}$ to be at least $2$, which is satisfied as the last two rows of the product are $\begin{pmatrix}1/2 & 0\end{pmatrix}$ and $\begin{pmatrix}0 & 1/2\end{pmatrix}$.

    $(4B):$ We need the rank of $P_B \cdot \begin{pmatrix}
        \begin{matrix}
            1 & 0\\
            0 & 1
        \end{matrix}\\
        -\mathcal{B}_2^{-1} \cdot \mathcal{B}_1
    \end{pmatrix}$ to be at least $2$, which also holds as the second and the last rows of the product are respectively $\begin{pmatrix}0 & d_{4}\end{pmatrix}$ and $\begin{pmatrix}1/2 & 0\end{pmatrix}$ (and $d_4 \neq 0$ from our assumptions).
\end{proof}

\begin{lemma}
\label{zero_column_lemma}
    If $Q(x_1, \dots, x_8)=Q'(x_1, \dots, x_7)+q_{88}x_8^2$, then $Q$ satisfies Theorem \ref{the_main_theorem}.
\end{lemma}

\begin{proof}
    By Lemma \ref{diagonal_submatrix_of_rank_lemma}, we may assume $\rank(Q_{1:5; \ 1:5}) \geqslant 4$.

    Let us note that if $Q_{1:5; \ 6:7}=0$, the conditions of Lemma \ref{the_main_tool_lemma} hold trivially for $A=(x_1, \dots, x_5), \ B=(x_6, x_7), \ C=(x_8)$, so we may assume $\rank(Q_{1:5; \ 6:7}) \geqslant 1$.

    We also note that if $\rank(Q_{1:5; \ 6:7})=1$ (without loss of generality $Q_{1:5; \ 6} \neq 0$), we can introduce a new variable $v=2\sum_{i=1}^5 q_{i6}$ and divide the variables as: $A=(x_1, \dots, x_5), \ B=(x_6, x_7, v), \ C=(x_8)$. Checking that the conditions of Lemma \ref{the_main_tool_lemma} hold is analogous to how it is done below in the case $\rank(Q_{1:5; \ 6:7})=2$ but simpler, so will not be written down in detail here. Instead, let us pass to the case $\rank(Q_{1:5; \ 6:7})=2.$
    
    Let
    \begin{equation}
        v_6 = 2\sum_{i=1}^5 q_{i6} x_i + q_{66}x_6, \ \ v_7=2\sum_{i=1}^6 q_{i7} x_i + q_{77}x_7^2.
    \end{equation}
    We may assume $\begin{pmatrix}
        q_{46} & q_{56}\\ q_{47} & q_{57}
    \end{pmatrix}$ is invertible.

    We need to bound by $O(X^6L^{-K})$ the following:
    \begin{multline}
        \sum_{\substack{x \in [X]^8, \\ v_6, v_7 \ll X}} \int_{\substack{\alpha \in \mathfrak{m}}} \Lambda(x) \ e\Big( \alpha \big( Q_{1:5; \ 1:5}(x_1, \dots, x_5) + x_6v_6+x_7v_7+q_{88}x_8^2- N \big) \Big) d\alpha \\ \cdot \int_{\substack{\beta, \gamma \in [0,1]}} e \bigg( \beta\Big(v_6 - 2\sum_{i=1}^5 q_{i6} x_i - q_{66}x_6\Big) + \gamma\Big(v_7 - \sum_{i=1}^6 q_{i7} x_i - q_{77}x_7^2\Big) \bigg) d\beta d\gamma.
    \end{multline}

    Let us divide the variables as follows: $$A = (x_6, x_7, v_6, v_7 ), \ B = (x_1, x_2, x_3, x_4, x_5 ), \ C = (x_8).$$
    Let us now verify that the conditions of Lemma \ref{the_main_tool_lemma} are satisfied.

    $(1):$ This condition is satisfied with $$P_A = \begin{pmatrix}
        0 & 0 & 1/2 & 0\\
        0 & 0 & 0 & 1/2\\
        1/2 & 0 & 0 & 0\\
        0 & 1/2 & 0 & 0
    \end{pmatrix}, \ P_B = Q_{1:5; \ 1:5}, \ c=q_{88}.$$
    $(2):$ Since $Q$ is invertible, $q_{88} \neq 0$.

    $(3A):$ We have $\mathcal{A} = \begin{pmatrix}
        q_{66} & 0 & -1 & 0\\
        2q_{67} & q_{77} & 0 & -1\\
    \end{pmatrix}$, so $\mathcal{A}_1=\begin{pmatrix}
        q_{66} & 0\\
        2q_{67} & q_{77}\\
    \end{pmatrix}$ and $\mathcal{A}_2= \begin{pmatrix}
        -1 & 0\\ 0 & -1
    \end{pmatrix}$. Hence, $\mathcal{A}_2$ is clearly invertible.

    $(3B):$ We have $\mathcal{B} = 2\begin{pmatrix}
        q_{16} & q_{26} & q_{36} & q_{46} & q_{56}\\
        q_{17} & q_{27} & q_{37} & q_{47} & q_{57}\\
    \end{pmatrix}$, so $\mathcal{B}_1=2\begin{pmatrix}
        q_{16} & q_{26} & q_{36}\\
        q_{17} & q_{27} & q_{37}\\
    \end{pmatrix}$, $\mathcal{B}_2=2\begin{pmatrix}
        q_{46} & q_{56}\\ q_{47} & q_{57}
    \end{pmatrix}$ and $\mathcal{B}_2$ is invertible from the assumptions.

    $(4A):$ We need the rank of $\begin{pmatrix}
        0 & 0 & 1/2 & 0\\
        0 & 0 & 0 & 1/2\\
        1/2 & 0 & 0 & 0\\
        0 & 1/2 & 0 & 0
    \end{pmatrix} \cdot \begin{pmatrix}
        \begin{matrix}
            1 & 0\\
            0 & 1\\
        \end{matrix}\\
        -\mathcal{A}_2^{-1} \cdot \mathcal{A}_1
    \end{pmatrix}$ ti be at least $2$, which is clearly the case.

    $(4B): $ We need to show that the rank of
    \begin{equation*}
        Q_{1:5; \ 1:5} \cdot  \begin{pmatrix}
        1 & 0 & 0\\
        0 & 1 & 0\\
        0 & 0 & 1\\
        a_{41} & a_{42} & a_{43}\\
        a_{51} & a_{52} & a_{53}
    \end{pmatrix}
    \end{equation*}
    is at least $2$, where
    $\begin{pmatrix}
        a_{41} & a_{42} & a_{43}\\
        a_{51} & a_{52} & a_{53}
    \end{pmatrix} = - \begin{pmatrix}
        q_{46} & q_{56}\\ q_{47} & q_{57}
    \end{pmatrix}^{-1} \cdot \begin{pmatrix}
        q_{16} & q_{26} & q_{36}\\ q_{17} & q_{27} & q_{37}
    \end{pmatrix}$. We note that if the rank of the given product is at most $1$, then we have (for $u_i=\begin{pmatrix}
        q_{1i} & q_{2i} & q_{3i} & q_{4i} & q_{5i}
    \end{pmatrix}^T$, the $i^{\text{th}}$ column $Q_{1:5; \ 1:5}$): $u_i+a_{4i}u_4+a_{5i}u_5=c_iv$ for $1 \leqslant i \leqslant 3$, a constant $c_i$ and some column vector $v$ of length $5$. Hence, $\rank \big(Q_{1:5; \ 1:5}\big) = \rank \big( u_1, \dots, u_5 \big) \leqslant 3$, contrary to the assumption that $\rank (Q_{1:5; \ 1:5}) \geqslant 4$.
\end{proof}

\begin{lemma}
\label{three_columns_rank_1}
    If $Q$ is of the form:
    $\begin{pmatrix}
        \xi\xi^T + D & \xi \gamma^T\\
        \gamma\xi^T & E
    \end{pmatrix}$
    for $\xi \in M_{3 \times 1}(\q), \ \gamma \in M_{5 \times 1}(\q), \ D \in M_{3 \times 3}(\q)$ diagonal and $E \in M_{5 \times 5}(\z)$, then it satisfies Theorem \ref{the_main_theorem}.
\end{lemma}

\begin{remark}
    We note that the condition for the form of $Q$ is satisfied (for $Q$ symmetric) if and only if the first three rows of $Q$, after adjusting their diagonal entries, have rank at most $1$.
\end{remark}

\begin{proof}
    Let $\xi=\begin{pmatrix}a_1 & a_2 & a_3\end{pmatrix}^T$, $\gamma=\begin{pmatrix}b_1 & b_2 & b_3 & b_4 & b_5\end{pmatrix}^T$ and let us denote the diagonal entries of $D$ by: $d_1, d_2, d_3$. 

    Since $Q$ is invertible, at least two of the $d_i$s must be nonzero. Let $d_1, d_2 \neq 0$. We also note that if $a_i=0$ for some $1 \leqslant i \leqslant 3$, we can apply Lemma \ref{zero_column_lemma}, so we may assume $a_1, a_2, a_3 \neq 0.$

    Finally, let us notice that if $E$ is diagonal, we can apply Lemma \ref{five_five_zero_lemma}, so we may assume it has a nonzero diagonal entry. Let $q_{78} \neq 0$.

    For simplicity, let:
    \begin{equation}
        y_i=x_i \ \text{for} \ 1 \leqslant i \leqslant 3, \ \ z_i=x_{i} \ \text{for} \ 4 \leqslant i \leqslant 7, \ \ t=x_8, \ \ v=2\sum_{i=1}^7 q_{i8}x_i+q_{88}x_8, \ \ u=\sum_{i=1}^3a_ix_i.
    \end{equation}
    We need to bound by $O(X^6L^{-K})$ the following:
    \begin{multline}
        \sum_{\substack{y \in [X]^3, z \in [X]^4,\\ t \in [X], \ u, v \ll X}} \Lambda(x) \int_{\alpha \in \mathfrak{m}} e \bigg( \alpha \Big( z^T Q_{4:7; 4:7} z + \sum_{i=1}^3 d_iy_i^2 + u^2 + 2u\sum_{i=4}^7 b_iz_i + vt \Big) \bigg) \cdot \\
        \cdot \int_{\beta \in [0, 1]} e \bigg( \beta \Big( u - \sum_{i=1}^3a_iy_i \Big) \bigg) \int_{\gamma \in [0, 1]} e \bigg( \gamma \Big( v - 2\sum_{i=1}^3 q_{i8}y_i - 2\sum_{i=4}^7 q_{i8}z_i+q_{88}t \Big) \bigg).
    \end{multline}

    Now, let: $$A=(y_2, t, y_3, v), \ B=(z_1, z_2, z_3, z_4, u), C=(y_1).$$ We will check that this grouping of variables satisfies the conditions of Lemma \ref{the_main_tool_lemma}.

    $(1)$: It is satisfied, with $P_A=\begin{pmatrix}
        d_2 & 0 & 0 & 0\\
        0 & 0 & 0 & 1/2\\
        0 & 0 & d_3 & 0\\
        0 & 1/2 & 0 & 0\\
    \end{pmatrix}$, $P_B=\begin{pmatrix}
        e_{11} & e_{12} & e_{13} & e_{14} & b_1\\
        e_{21} & e_{22} & e_{23} & e_{24} & b_2\\
        e_{31} & e_{32} & e_{33} & e_{34} & b_3\\
        e_{41} & e_{42} & e_{43} & e_{44} & b_4\\
        b_1 & b_2 & b_3 & b_4 & 1
    \end{pmatrix}$ and $c=d_1$.

    $(2)$: It is satisfied, since $d_1 \neq 0$.

    $(3A)$: We have $\mathcal{A}=\begin{pmatrix}
        a_2 & 0 & a_3 & 0\\
        2b_5a_2 & q_{88} & 2b_5a_3 & -1\\
    \end{pmatrix}$, so $\mathcal{A}_2=\begin{pmatrix}
        a_3 & 0\\
        2b_5a_3 & -1\\
    \end{pmatrix}$, which is invertible since (as we assumed) $a_3 \neq 0$.

    $(3B)$: We have: $\mathcal{B}=\begin{pmatrix}
        0 & 0 & 0 & 0 & -1\\
        2q_{48} & 2q_{58} & 2q_{68} & 2q_{78} & 0\\
    \end{pmatrix}$, so $\mathcal{B}_2=\begin{pmatrix}
        0 & -1\\
        2q_{78} & 0\\
    \end{pmatrix}$, which is invertible since $q_{78} \neq 0$.

    $(4A)$: We need to show that the rank of the following is at least $2$:
    \begin{equation*}
        \begin{pmatrix}
        d_2 & 0 & 0 & 0\\
        0 & 0 & 0 & 1/2\\
        0 & 0 & d_3 & 0\\
        0 & 1/2 & 0 & 0\\
    \end{pmatrix} \cdot \begin{pmatrix}
        \begin{matrix}
            1 & 0\\
            0 & 1\\
        \end{matrix}\\
        -\mathcal{A}_2^{-1} \cdot \mathcal{A}_1
    \end{pmatrix},
    \end{equation*}
    which is clear, as the first row of the product is $\begin{pmatrix}
        d_2 & 0
    \end{pmatrix}$ and the last row is $\begin{pmatrix}0 & 1/2\end{pmatrix}$.

    $(4B)$: We have:
    \begin{equation*}
        -\mathcal{B}_2^{-1} \cdot \mathcal{B}_1 =\begin{pmatrix}
        0 & -\frac{1}{2q_{78}}\\
        1 & 0\\
    \end{pmatrix} \cdot \begin{pmatrix}
        0 & 0 & 0\\
        2q_{48} & 2q_{58} & 2q_{68}\\
    \end{pmatrix}=\begin{pmatrix}
        -\frac{q_{48}}{q_{78}} & -\frac{q_{58}}{q_{78}} & -\frac{q_{68}}{q_{78}}\\
        0 & 0 & 0
    \end{pmatrix}.
    \end{equation*}
    Hence, we need to show that the rank of the following is at least $2$: 
    \begin{equation*}
        \begin{pmatrix}
        e_{11} & e_{12} & e_{13} & e_{14} & b_1\\
        e_{21} & e_{22} & e_{23} & e_{24} & b_2\\
        e_{31} & e_{32} & e_{33} & e_{34} & b_3\\
        e_{41} & e_{42} & e_{43} & e_{44} & b_4\\
        b_1 & b_2 & b_3 & b_4 & 1
    \end{pmatrix} \cdot \begin{pmatrix}
        1 & 0 & 0\\
        0 & 1 & 0\\
        0 & 0 & 1\\
        -\frac{q_{48}}{q_{78}} & -\frac{q_{58}}{q_{78}} & -\frac{q_{68}}{q_{78}}\\
        0 & 0 & 0
    \end{pmatrix}.
    \end{equation*}
    If that product had rank at most $1$, then in particular (denoting by $v_i := Q_{3:7; i}$) we would have $\rank\big( \langle v_1-\frac{q_{48}}{q_{78}}v_4, \ v_2-\frac{q_{58}}{q_{78}}v_4, \ v_3-\frac{q_{68}}{q_{78}}v_4 \rangle \big) \leqslant 1$, giving us two linear relations on these vectors, so we would have $\rank(Q_{3:7; \ 4:7}) \leqslant 2$. However, $8=\rank(Q) \leqslant \rank(Q_{1:3; [8]})+\rank(Q_{8; [8]})+\rank(Q_{8; 4:7})+\rank(Q_{1:7; 4:7})$. Hence, $\rank(Q_{3:7; 4:7})=\rank(Q_{1:7; 4:7}) \geqslant 8 - 3 - 1 - 1 = 3$. Hence, the rank of the product must be at least $2$ and so the condition $(4B)$ is also satisfied.
\end{proof}

\section{The possible structures of \texorpdfstring{$Q$}{Q}}
\label{section_structure}

\begin{definition}
    Let $M$ be an off-diagonal, invertible $3 \times 3$ submatrix of $Q$, i.e. for some $i_1, i_2, i_3, j_1, j_2, j_3$ pairwise different we have
    \begin{equation*}
        M=\begin{pmatrix}
        q_{i_1j_1} & q_{i_1j_2} & q_{i_1j_3}\\
        q_{i_2j_1} & q_{i_2j_2} & q_{i_2j_3}\\
        q_{i_3j_1} & q_{i_3j_2} & q_{i_3j_3}\\
    \end{pmatrix}.
    \end{equation*}
    Let $\{k,l\}=[8] \setminus \{i_1, i_2, i_3, j_1, j_2, j_3\}$, $$H_M:=\begin{pmatrix}
        q_{i_1l} & q_{i_1k} & q_{i_1j_1} & q_{i_1j_2} & q_{i_1j_3}\\
        q_{i_2l} & q_{i_2k} & q_{i_2j_1} & q_{i_2j_2} & q_{i_2j_3}\\
        q_{i_3l} & q_{i_3k} & q_{i_3j_1} & q_{i_3j_2} & q_{i_3j_3}\\
    \end{pmatrix}$$ and $H_{M, k}$ be $H_M$ without the column indexed by $j_k$ for $1 \leqslant k \leqslant 3$. We note that since $M$ is invertible, we have $2 \leqslant \rank(H_{M, 1}), \rank(H_{M, 2}), \rank(H_{M, 3}) \leqslant 3$. If exactly $a$ of $\rank(H_{M, 1}), \rank(H_{M, 2}),$ $\rank(H_{M, 3})$ equal $3$, we say that $M$ is `horizontally of type $a$'.

    We note that we can analogously define the `vertical type of $M$', \\ taking $V_M:=\begin{pmatrix}
        q_{i_1j_1} & q_{i_2j_1} & q_{i_3j_1} & q_{kj_1} & q_{lj_1}\\
        q_{i_1j_2} & q_{i_2j_2} & q_{i_3j_2} & q_{kj_2} & q_{lj_2}\\
        q_{i_1j_3} & q_{i_2j_3} & q_{i_3j_3} & q_{kj_3} & q_{lj_3}\\
    \end{pmatrix}^T$ instead of $H_M$.
\end{definition}

Since the off-diagonal rank of $Q$ is $3$, it has a $3 \times 3$ invertible submatrix that does not contain its diagonal entries. Without loss of generality it is the top-right $3 \times 3$ submatrix $M=Q_{1:3; \ 6:8}$.
We may assume $\rank(H_{M, 8}) \leqslant \rank(H_{M, 7}) \leqslant \rank(H_{M, 6})$ (swapping variables $6,7,8$ as needed) and that $\rank(V_{M, 1}) \leqslant \rank(V_{M, 2}) \leqslant \rank(V_{M, 3})$ (swapping variables $1,2,3$).

The analysis in this section is based on the results in \cite{nine}, but we use a bit different notation and, in particular, assume that $\rank(H_{M,8}) \leqslant \rank(H_{M,7}) \leqslant \rank(H_{M,6})$ while Zhao assumes that $\rank(H_{M, 8}) \geqslant \rank(H_{M, 7}) \geqslant \rank(H_{M, 6})$.

Zhao showed the following four lemmas (\cite[5.13, 5.11, 5.10, 5.9]{nine}):

\begin{lemma}
\label{zero_twos_lemma}
    If $M$ is of horizontal type $3$, then $Q$ has the form: $$\begin{pmatrix}
        E & UC\\ C^TU^T & D+C^THC
    \end{pmatrix}$$
    for $E \in M_{3 \times 3}(\z)$, $H \in M_{3 \times 3}(\q)$, $U \in GL_3(\q)$, $C \in M_{3 \times 5}(\z)$ with $\rank(C)=3$ and $D \in M_{5 \times 5}(\q)$ diagonal.
\end{lemma}

Theorem \ref{the_main_theorem} for this case is proven in Section \ref{section_three_three}.

\begin{lemma}
\label{one_two_lemma}
    If $M$ is of horizontal type $2$, then $Q$ has the form:
    $$\begin{pmatrix}
        E & UC & \gamma\\
        C^TU^T & D+C^THC & \xi \\
        \gamma^T & \xi^T & f
    \end{pmatrix}$$
    for $E \in M_{3 \times 3}(\z)$, $\gamma \in M_{3 \times 1}(\z), \ \xi \in M_{4 \times 1}(\z)$, $f \in \mathbb{Z}$, $U \in M_{3 \times 2}(\q)$, $C \in M_{2 \times 4}(\z)$ with $\rank(C)=2$, $H \in M_{2 \times 2}(\q)$ and $D \in M_{4 \times 4}(\q)$ diagonal. Moreover, $\begin{pmatrix}
        U & \gamma
    \end{pmatrix} \in GL_3(\q)$.
\end{lemma}

Theorem \ref{the_main_theorem} for this case is proven in Section \ref{section_three_two}.

\begin{lemma}
\label{two_twos_lemma}
    If $M$ is of horizontal type $1$, then $Q$ has the form:
    $$\begin{pmatrix}
        E & a b^T & C\\
        ba^T & D+hbb^T & G\\
        C^T & G^T & F
    \end{pmatrix}$$
    for $F \in M_{3 \times 3}(\z)$, $G, C \in M_{3 \times 2}(\z)$, $E \in M_{2 \times 2}(\z)$, $D\in M_{2 \times 2}(\q)$ diagonal, $h \in \q$ and $a, b \in M_{3 \times 1}(\q)$ with $\begin{pmatrix}
        b_3 a & C
    \end{pmatrix} \in GL_3(\q)$.
\end{lemma}

\begin{lemma}
\label{three_twos_lemma}
    If $M$ is of horizontal type $0$, then $Q$ has the form:
    $$\begin{pmatrix}
        E & 0 & U\\
        0 & D & C\\
        U^T & C^T & F
    \end{pmatrix}$$
    for $E, F \in M_{3 \times 3}(\z)$, $U \in GL_3(\z)$, $C \in M_{2 \times 3}(\z)$ and $D \in M_{2 \times 2}(\z)$ diagonal.
\end{lemma}

\begin{remark}
    We note that both Lemma \ref{two_twos_lemma} and Lemma \ref{three_twos_lemma} imply $\rank(Q_{1:3; \ 4:6})=1$.
\end{remark}

Theorem \ref{the_main_theorem} for the last two cases is proven in Section \ref{section_twos}.

We note that if $\rank(V_{M,1})=\rank(V_{M,2})=\rank(V_{M,3})=3$, then $Q$ has the structure as in Lemma \ref{zero_twos_lemma}, reflected in its diagonal, and similarly for the other structural lemmas.

In the next two sections, we prove Theorem \ref{the_main_theorem} in the case where, for some arrangement of variables, we have the case from Lemma \ref{zero_twos_lemma} (Section \ref{section_three_three}) or from Lemma \ref{one_two_lemma} (Section \ref{section_three_two}). In the final section, we will prove the theorem in the case that for no arrangement of variables we have a situation from Lemma \ref{zero_twos_lemma} or from Lemma \ref{one_two_lemma}.

\section{The case \texorpdfstring{$M$}{M} has an off-diagonal submatrix of type 3 (vertically or horizontally).}
\label{section_three_three}

In this section, we prove Theorem \ref{the_main_theorem} in the following case:

\begin{proposition}
\label{proposition_3}
    Let $Q$ be an invertible, symmetrix matrix of off-diagonal rank $3$, encoding a quadratic form in $8$ variables.

    If $Q$ has an off-diagonal, invertible submatrix $M$ of (vertical or horizontal) type $3$, then we can bound its minor arc: $$\int_{\alpha \in \mathfrak{m}} \sum_{x \in [X]^8} \Lambda(x) e \big( \alpha(Q(x) - N) \big) \ d\alpha = O(X^6L^{-K}).$$
\end{proposition}

We recall the structure of $Q$ from Lemma \ref{zero_twos_lemma}. Since $C^THC$ is symmetric for $C$ a $3 \times 5$ matrix (and $5 \geqslant 3$), $H$ is also symmetric. As it also has rational entries, it is diagonalisable. We may assume that $H$ is diagonal (otherwise, if $H=P^TH'P$ for $H'$ diagonal and $P^TP=Id_3$, we can substitute $C'=PC, U'=UP^T$). Let us note that the case when $H=0$ has already been handled in Lemma \ref{five_five_zero_lemma}. Hence, from now on we assume that $H \neq 0$. Let $H=\begin{pmatrix}
    h_1 & 0 & 0\\
    0 & h_2 & 0\\
    0 & 0 & h_3
\end{pmatrix}$, where, without loss of generality, $h_1 \neq 0$. Since $U$ is invertible, we have $u_{i1} \neq 0$ for some $i$. By permuting $(x_1, x_2, x_3)$ if necessary, we may assume $u_{31} \neq 0$.

Let $d_1, \dots, d_5$ denote the diagonal entries of $D$. We will show that we may permute variables $x_4, x_5, x_6, x_7, x_8$ so that:
\begin{itemize}
    \item $d_1, d_2 \neq 0$,
    \item $\rank(C_{1:3; \ 3:5})=3$.
\end{itemize}
Let $J=\{1 \leqslant j \leqslant 5: d_{j}=0\}$. Since $8=\rank(Q) \leqslant 3+\rank(C)+\rank(D)\leqslant 6+\rank(D)$, so $2 \leqslant \rank(D)$ and $|J| \leqslant 3$. We note that the columns of $C_{1:3; \ J}$ are linearly independent, as the columns of $Q_{1:8; \ J}$ are linearly independent. Hence, since $\rank(C)=3$, there exists $K \subset \{1, \dots, 5\}$, $|K|=3-|J|$ such that $\rank(C_{1:3; \ K \cup J})=3$. Then, permuting the variables so that $K \cup J=\{3,4,5\}$, we get a permutation satisfying $d_{1}, d_{2} \neq 0$ and $\rank(C_{1:3; \ 3:5})=3$.

Let us denote:
\begin{equation}
    z = \begin{pmatrix}x_1 & x_2 & x_3\end{pmatrix}^T, \ \ y = \begin{pmatrix}x_4 & x_5 & x_6 & x_7 & x_8\end{pmatrix}^T, \ \ v = \begin{pmatrix}v_1 & v_2 & v_3\end{pmatrix}^T = Cy + Sz
\end{equation}
for $S$ such that $s_{ij}=1_{h_i \neq 0} \cdot \frac{u_{ji}}{h_i}$ (we assume that the value of this expression is $0$ also if $h_i=0$). Let us note that since $u_{31} \neq 0$, we also have $s_{13} \neq 0$. Let $F=E+S^THS-2US$.

We need to bound by $O(X^6L^{-K})$ the following:
\begin{multline*}
    \int_{\alpha \in \mathfrak{m}} \sum_{x \in [X]^8} \Lambda(x) e\big(\alpha (x^TQx-N)\big) \ d \alpha \\ = \int_{\alpha \in \mathfrak{m}} \sum_{\substack{y \in [X]^5, \\ z \in [X]^3}} \Lambda(y)\Lambda(z)e \Big(\alpha\big(z^TEz + 2z^TUCy + y^T(C^THC+D)y-N\big) \Big),
\end{multline*}
which equals
\begin{multline*}
    \sum_{\substack{y \in [X]^5, \\ z \in [X]^3, \\ v_1, v_2, v_3 \ll X}} \Lambda(y)\Lambda(z)\\ \cdot \int_{\substack{\alpha \in \mathfrak{m}}} e \Big( \alpha \big( z^T(E+S^THS-2US)z+ 2z^T(U-S^TH)v+v^THv+y^TDy-N \big) \Big) \ d\alpha \\ \cdot \int_{\substack{\beta_1^3 \in [0, 1]}} e \bigg( \sum_{i=1}^3\beta_i \Big( v_i - \sum_{j=1}^5c_{ij}y_j-\sum_{j=1}^3s_{ij}z_j \Big) \bigg) \ d\beta_i.
\end{multline*}

We note that $U-S^TH=\begin{pmatrix}
    0 & u_{12} \cdot 1_{h_2 = 0} & u_{13} \cdot 1_{h_3 = 0}\\
    0 & u_{22} \cdot 1_{h_2 = 0} & u_{23} \cdot 1_{h_3 = 0}\\
    0 & u_{32} \cdot 1_{h_2 = 0} & u_{33} \cdot 1_{h_3 = 0}\\
\end{pmatrix}$ from the choice of $S$, so the products $v_1z_i$ do not appear in the exponent.

Now, let us take: $$A=(y_1, y_2, y_3, y_4, y_5), \ B=(z_1, z_2, z_3, v_2, v_3), \ C=(v_1).$$ We now show that such $A, B, C$ satisfy the preconditions of Lemma \ref{the_main_tool_lemma}.

$(1):$ This condition is satisfied for:
\begin{equation*}
    P_A=\begin{pmatrix}
    d_{1} & 0 & 0 & 0 & 0\\
    0 & d_{2} & 0 & 0 & 0\\
    0 & 0 & d_{3} & 0 & 0\\
    0 & 0 & 0 & d_{4} & 0\\
    0 & 0 & 0 & 0 & d_{5}\\
\end{pmatrix},
\end{equation*}
\begin{equation*}
    P_B=\begin{pmatrix}
    f_{11} & f_{12} & f_{13} & u_{12} \cdot 1_{h_2 = 0} & u_{13} \cdot 1_{h_3 = 0}\\
    f_{21} & f_{22} & f_{23} & u_{22} \cdot 1_{h_2 = 0} & u_{23} \cdot 1_{h_3 = 0}\\
    f_{31} & f_{32} & f_{33} & u_{32} \cdot 1_{h_2 = 0} & u_{33} \cdot 1_{h_3 = 0}\\
    u_{12} \cdot 1_{h_2 = 0} & u_{22} \cdot 1_{h_2 = 0} & u_{32} \cdot 1_{h_2 = 0} & h_2 & 0\\
    u_{13} \cdot 1_{h_3 = 0} & u_{23} \cdot 1_{h_3 = 0} & u_{33} \cdot 1_{h_3 = 0} & 0 & h_3\\
\end{pmatrix}
\end{equation*}
and $c=h_1.$

$(2):$ We assumed that $h_1 \neq 0$.

$(3A):$ We have $\mathcal{A}=C$, so $\mathcal{A}_2=C_{1:5; \ 3:5}$. We may assume it is invertible.

$(3B):$ We have $\mathcal{B}=\begin{pmatrix}
    s_{11} & s_{12} & s_{13} & 0 & 0\\
    s_{21} & s_{22} & s_{23} & -1 & 0\\
    s_{31} & s_{32} & s_{33} & 0 & -1
\end{pmatrix}$, so $\mathcal{B}_1=\begin{pmatrix}
    s_{11} & s_{12}\\
    s_{21} & s_{22}\\
    s_{31} & s_{32}
\end{pmatrix}$, $\mathcal{B}_2=\begin{pmatrix}
    s_{13} & 0 & 0\\
    s_{23} & -1 & 0\\
    s_{33} & 0 & -1
\end{pmatrix}$ so, since $s_{13}\neq0$, $\mathcal{B}_2$ is invertible.

$(4A):$ We need the rank of $\begin{pmatrix}
    d_{1} & 0 & 0 & 0 & 0\\
    0 & d_{2} & 0 & 0 & 0\\
    0 & 0 & d_{3} & 0 & 0\\
    0 & 0 & 0 & d_{4} & 0\\
    0 & 0 & 0 & 0 & d_{5}\\
\end{pmatrix}
\cdot
\begin{pmatrix}
    \begin{matrix}
        1 & 0\\
        0 & 1
    \end{matrix}\\
    -C_{1:3; \ 3:5}^{-1} \cdot C_{1:3; \ 1:2}
\end{pmatrix}$ to be at least $2$, which is the case as the first two rows of the product are $\begin{pmatrix}d_1 & 0\end{pmatrix}$ and $\begin{pmatrix}0 & d_2\end{pmatrix}$.

$(4B):$ We note that:
\begin{multline*}
    -\mathcal{B}_2^{-1} \cdot \mathcal{B}_1=\begin{pmatrix}
    -\frac{s_{11}}{s_{13}} & -\frac{s_{12}}{s_{13}}\\
    -\frac{s_{23}s_{11}}{s_{13}}+s_{21} & \frac{s_{12}s_{23}}{s_{13}}+s_{22}\\
    -\frac{s_{11}s_{33}}{s_{13}}+s_{31} & \frac{s_{12}s_{33}}{s_{13}}+s_{32}
\end{pmatrix}\\
=\begin{pmatrix}
    -\frac{u_{11}}{u_{31}} & -\frac{u_{21}}{u_{31}}\\
    -\frac{u_{32}u_{11}}{h_2u_{31}} \cdot 1_{h_2 \neq 0}+\frac{u_{12}}{h_2} \cdot 1_{h_2 \neq 0} & \frac{u_{21}u_{32}}{h_2u_{31}} \cdot 1_{h_2 \neq 0}+\frac{u_{22}}{h_2} \cdot 1_{h_2 \neq 0}\\
    -\frac{u_{11}u_{33}}{h_3u_{31}}\cdot 1_{h_3 \neq 0}+\frac{u_{13}}{h_3} \cdot 1_{h_3 \neq 0} & \frac{u_{21}u_{33}}{h_3u_{31}}\cdot 1_{h_3 \neq 0}+\frac{u_{23}}{h_3} \cdot 1_{h_3 \neq 0}
\end{pmatrix}.
\end{multline*}
We need to show that the rank of $P_B \cdot \begin{pmatrix}
    \begin{matrix}
        1 & 0\\ 0 & 1
    \end{matrix}\\
    -\mathcal{B}_2^{-1} \cdot \mathcal{B}_1
\end{pmatrix}$ is at least $2$. We in particular show that the two bottom rows of this product are nonzero and linearly independent, calculating a $2 \times 2$ matrix of these rows:
\begin{multline*}
\begin{pmatrix}
    u_{12} \cdot 1_{h_2 = 0} & u_{22} \cdot 1_{h_2 = 0} & u_{32} \cdot 1_{h_2 = 0} & h_2 & 0\\
    u_{13} \cdot 1_{h_3 = 0} & u_{23} \cdot 1_{h_3 = 0} & u_{33} \cdot 1_{h_3 = 0} & 0 & h_3\\
\end{pmatrix}\\
\cdot \begin{pmatrix}
    1 & 0\\
    0 & 1\\
    -\frac{u_{11}}{u_{31}} & -\frac{u_{21}}{u_{31}}\\
    -\frac{u_{32}u_{11}}{h_2 u_{31}} \cdot 1_{h_2 \neq 0}+\frac{u_{12}}{h_2} \cdot 1_{h_2 \neq 0} & -\frac{u_{21}u_{32}}{h_2u_{31}} \cdot 1_{h_2 \neq 0}+\frac{u_{22}}{h_2} \cdot 1_{h_2 \neq 0}\\
    -\frac{u_{11}u_{33}}{h_3u_{31}} \cdot 1_{h_3 \neq 0}+\frac{u_{13}}{h_3} \cdot 1_{h_3 \neq 0} & -\frac{u_{21}u_{33}}{h_3u_{31}} \cdot 1_{h_3 \neq 0}+\frac{u_{23}}{h_3} \cdot 1_{h_3 \neq 0}
\end{pmatrix}\\
= \frac{1}{u_{31}}\begin{pmatrix}
    (u_{12}u_{31}-u_{11}u_{32}) \cdot (1_{h_2 = 0} + 1_{h_2 \neq 0}) & (u_{22}u_{31}-u_{32}u_{21}) \cdot (1_{h_2 = 0} + 1_{h_2 \neq 0})\\
    (u_{13}u_{31}-u_{11}u_{33}) \cdot (1_{h_3 = 0} + 1_{h_3 \neq 0}) &
    (u_{23}u_{31}-u_{21}u_{33}) \cdot (1_{h_3 = 0} + 1_{h_3 \neq 0})
\end{pmatrix}\\
= \frac{1}{u_{31}}\begin{pmatrix}
    u_{12}u_{31}-u_{11}u_{32} & u_{22}u_{31}-u_{32}u_{21}\\
    u_{13}u_{31}-u_{11}u_{33} & u_{23}u_{31}-u_{21}u_{33}
\end{pmatrix}.
\end{multline*}
We note that the determinant of this matrix is $\frac{1}{u_{31}}\det(U) \neq 0$, so these two rows are non-zero linearly independent, so the rank of the matrix in question is $2$ and hence $A$ satisfies $(4)$.

Hence, all the conditions of Lemma \ref{the_main_tool_lemma} are satisfied, and applying the lemma finishes the proof in this case.

\section{The case \texorpdfstring{$M$}{M} has an off-diagonal submatrix of type 2 (vertically or horizontally).}
\label{section_three_two}

In this section, we prove the following: 
\begin{proposition}
\label{proposition_2}
    Let $Q$ be an invertible, symmetrix matrix of off-diagonal rank $3$, encoding a quadratic form in $8$ variables.

    If $Q$ has an off-diagonal, invertible submatrix $M$ of (vertical or horizontal) type $2$, then we can bound its minor arc: $$\int_{\alpha \in \mathfrak{m}} \sum_{x \in [X]^8} \Lambda(x) e \big( \alpha(Q(x) - N) \big) \ d\alpha = O(X^6L^{-K}).$$
\end{proposition}

Again, without loss of generality the top-right submatrix of $Q$ has vertical rank $3$. We recall its form from Lemma \ref{zero_twos_lemma}. As in the previous case, we may assume $H=\begin{pmatrix}h_1 & 0 \\ 0 & h_2 \end{pmatrix}$ is diagonal.

We will deal with the case $H=0$ separately later in Lemma \ref{case_233_with_h_0}, so for now let us assume that $h_1 \neq 0$.

We will now show that (permuting the variables $(x_4, x_5, x_6, x_7)$ if necessary) we may assume:
\begin{itemize}
    \item $d_{1}, d_{2} \neq 0$,
    \item $C_{1,2; \ 3, 4}$ is invertible.
\end{itemize}
Let $J=\{1 \leqslant j \leqslant 4: d_j=0\}$. We note that since $\rank(Q_{1:8; \ J})=J$ we have $\rank(C_{1,2; \ J})=J$, so $|J| \leqslant 2$. Since $\rank(C)=2$, we can choose $K \subset \{1, 2, 3, 4\}$, $|K|=2-|J|$ such that $\rank(C_{1,2; \ K \cup J})=2$. Now we can permute the variables $(x_4, x_5, x_6, x_7)$ so that $K \cup J=\{3, 4\}$. Then the listed conditions are satisfied.

Since from the structural lemma the two columns of $U$ and $\gamma$ are independent, let
\begin{equation*}
    V=\begin{pmatrix}
    u_{11} & u_{12} & \gamma_{1}\\
    u_{21} & u_{22} & \gamma_{2}\\
    u_{31} & u_{32} & \gamma_{3}\\
\end{pmatrix}\in GL_3
\end{equation*}
(this notation makes some calculations later seem more natural). Since $V$ is invertible, in particular its first and third columns are linearly independent, so (permuting variables $x_1, x_2, x_3$ as necessary) we may assume $\begin{pmatrix}
    v_{11} & v_{13} \\ v_{21} & v_{23}
\end{pmatrix}$ is invertible.

Let us denote:
\begin{equation}
    z=\begin{pmatrix}x_1 & x_2 & x_3\end{pmatrix}^T, \ \ y=\begin{pmatrix}x_4 & x_5 & x_6 & x_7\end{pmatrix}^T, \ \ t=x_8, \ \ r=\begin{pmatrix}r_1 \\ r_2 \end{pmatrix} = Cy+Sz,
\end{equation}
where $S$ is a $2 \times 3$ matrix such that $s_{ij}=1_{h_i \neq 0}\cdot\frac{u_{ji}}{h_i}$ (where again, we assume that the value of this expression is $0$ when $h_i=0$). Finally, let:
\begin{equation}
    F=E-2US+S^THS, \ \ w=\gamma^Tz+\xi^T y+f t.
\end{equation}
Now we need to bound by $O(X^6L^{-K})$ the following:
\begin{multline}
    \int_{\alpha \in \mathfrak{m}} \sum_{x \in [X]^8} \Lambda(x) e(\alpha (x^TQx-N)) d \alpha \\
    = \int_{\alpha \in \mathfrak{m}} \sum_{\substack{z \in [X]^3, \\ y \in [X]^4, \\ t \in [X]}} \Lambda(y)\Lambda(z)\Lambda(t) e \Big( \alpha \big( z^TEz +2z^TUCy+y^T(D+C^THC)y+t(\gamma^T z + \xi^T y + f t) -N\big) \Big) \ d\alpha \\
    = \sum_{\substack{z \in [X]^3, \ y \in [X]^4, \\t \in [X], \ r_1^2 \ll X, \\ w \ll X}} \Lambda(y)\Lambda(z)\Lambda(t) \int_{\alpha \in \mathfrak{m}} e \Big( \alpha \big(z^TFz + 2z^T (U-S^TH) r + r^THr + y^TDy + tw - N \big) \Big) \ d\alpha\\
    \cdot \int_{\beta_1^2, \gamma \in [0, 1]} e \bigg( \sum_{i=1}^2 \beta_i \Big( r_i - \sum_{j=1}^4 c_{ij} y_j - \sum_{j=1}^3 s_{ij} z_j \Big) + \gamma \Big( w - \gamma^Tz - \xi^T y - ft \Big) \bigg) \ d\beta_i \ d\gamma.
\end{multline}
We note that as $U-S^TH=\begin{pmatrix}0 & 1_{h_2 = 0} u_{12} \\ 0 & 1_{h_2 = 0} u_{22} \\ 0 & 1_{h_2 = 0} u_{32} \end{pmatrix}$, so $r_1z_i$ does not appear in the exponent above for any $i$.
Now let us take: $$A=(y_1, t, y_3, y_4, w), \ B=(z_3, y_2, z_1, z_2, r_2), \ C=(r_1).$$ We will now show that the preconditions of Lemma \ref{the_main_tool_lemma} are satisfied.

$(1):$ Let us take: $$P_A=\begin{pmatrix}
    d_1 & 0 & 0 & 0 & 0\\
    0 & 0 & 0 & 0 & 1/2\\
    0 & 0 & d_{3} & 0 & 0 \\
    0 & 0 & 0 & d_{4} & 0 \\
    0 & 1/2 & 0 & 0 & 0
\end{pmatrix}, \ P_B=\begin{pmatrix}
    f_{33} & 0 & f_{31} & f_{32} & 1_{h_2=0}v_{32} \\
    0 & d_2 & 0 & 0 & 0 \\
    f_{13} & 0 & f_{11} & f_{12} & 1_{h_2=0}v_{12} \\
    f_{23} & 0 & f_{21} & f_{22} & 1_{h_2=0}v_{22} \\
    1_{h_2=0}v_{32} & 0 & 1_{h_2=0}v_{12} & 1_{h_2=0}v_{22} & h_2 \\
\end{pmatrix}, \ c=h_1.$$
$(2):$ We assumed that $h_1 \neq 0$ (the other case is dealt with in Lemma \ref{case_233_with_h_0}).

$(3A):$ We have $\mathcal{A} = \begin{pmatrix}
    c_{11} & 0 & c_{13} & c_{14} & 0 \\
    c_{21} & 0 & c_{23} & c_{24} & 0 \\
    \xi_{1} & f & \xi_{3} & \xi_{4} & -1 
\end{pmatrix}$, so $\mathcal{A}_2$ is invertible since, as we showed we may assume, $C_{1,2; \ 3,4}$ is invertible.

$(3B):$ We have $\mathcal{B}=\begin{pmatrix}
\frac{v_{31}}{h_1} & c_{12} & \frac{v_{11}}{h_1} & \frac{v_{21}}{h_1} & 0 \\
1_{h_2 \neq 0} \cdot \frac{v_{32}}{h_2} & c_{22} & 1_{h_2 \neq 0} \cdot \frac{v_{12}}{h_2} & 1_{h_2 \neq 0} \cdot \frac{v_{22}}{h_2} & -1 \\
v_{33} & \xi_{2} & v_{13} & v_{23} & 0 \\
\end{pmatrix}$. Hence, $\mathcal{B}_2$ is invertible as we earlier assumed that $v_{23}v_{11}-v_{21}v_{13} \neq 0$).

$(4A):$ We need the rank of $P_A \cdot \begin{pmatrix}
    \begin{matrix}
        1 & 0\\
        0 & 1
    \end{matrix}\\
    -\mathcal{A}_2^{-1} \cdot \mathcal{A}_1
\end{pmatrix}$ ti be at least $2$, which is the case since its first row is $\begin{pmatrix}d_1 & 0\end{pmatrix}$ and its fifth row is $\begin{pmatrix}0 & 1/2 \end{pmatrix}$.

$(4B):$ We need to show that the rank of $P_B \cdot \begin{pmatrix}
    \begin{matrix}
        1 & 0 \\
        0 & 1
    \end{matrix}\\
    -\mathcal{B}_2^{-1} \cdot \mathcal{B}_1
\end{pmatrix}$ is at least $2$. The second row of the product is $\begin{pmatrix}0 & d_2\end{pmatrix}$, so we only need to show that the first column of the product has a nonzero entry. Noting that $\mathcal{B}_2^{-1}=\frac{h_1}{v_{11}v_{23}-v_{21}v_{13}}\begin{pmatrix}
        v_{23} & 0 & -\frac{v_{21}}{h_1} \\
        -v_{13} & 0 & \frac{v_{11}}{h_1} \\
        1_{h_2 \neq 0} \frac{v_{12}v_{23}-v_{22}v_{13}}{h_2} & -\frac{v_{11}v_{23}-v_{21}v_{13}}{h_1} & 1_{h_2 \neq 0}\frac{v_{11}v_{22}-v_{12}v_{21}}{h_1h_2}
    \end{pmatrix}$, we can calculate that the fifth entry of the first column is:
    \begin{multline*}
        -\frac{h_1}{v_{11}v_{23}-v_{21}v_{13}}\begin{pmatrix}1_{h_2=0}v_{32} \\ 0 \\ 1_{h_2=0}v_{12} \\ 1_{h_2=0}v_{22} \\ h_2 \end{pmatrix}^T \\ \cdot \begin{pmatrix}
    -\frac{v_{11}v_{23}-v_{21}v_{13}}{h_1}\\
    0\\
    \begin{pmatrix}
        v_{23} & 0 & -\frac{v_{21}}{h_1} \\
        -v_{13} & 0 & \frac{v_{11}}{h_1} \\
        1_{h_2 \neq 0} \frac{v_{12}v_{23}-v_{22}v_{13}}{h_2} & -\frac{v_{11}v_{23}-v_{21}v_{13}}{h_1} & 1_{h_2 \neq 0}\frac{v_{11}v_{22}-v_{12}v_{21}}{h_1h_2}
    \end{pmatrix} \cdot \begin{pmatrix}
        \frac{v_{31}}{h_1} \\ 1_{h_2 \neq 0} \frac{v_{32}}{h_2} \\ v_{33}
    \end{pmatrix}\\
    \end{pmatrix},
    \end{multline*}
    which equals $\frac{\det V}{v_{11}v_{23}-v_{21}v_{13}}$, which is nonzero from the assumptions.

Hence, all the preconditions of Lemma \ref{the_main_tool_lemma} are satisfied and so we may apply it to show that Theorem \ref{the_main_theorem} holds in this case as well.

Now we handle the special case $H=0$.

\begin{lemma}
\label{case_233_with_h_0}
    If $Q$ is of the form: $\begin{pmatrix}
        E & C & \gamma\\
        C^T & D & \xi\\
        \gamma^T & \xi^T & f
    \end{pmatrix}$, where $E \in M_{3 \times 3}(\z)$, $\gamma \in M_{3 \times 1}(\z)$,  $\xi \in M_{4 \times 1}(\z)$, $f \in \mathbb{Z}$, $C \in M_{3 \times 4}(\z)$ with $\rank(C) = 2$ and $\rank(C \ \gamma) = 3$ and $D \in M_{4 \times 4}(\z)$ diagonal, then $Q$ satisfies Theorem \ref{the_main_theorem}.
\end{lemma}

We start with a lemma that allows us to assume some of the values to be nonzero.

\begin{lemma}
\label{structure_lemma_233_case}
    If $Q$ is as in Lemma \ref{case_233_with_h_0}, we can permute variables $x_1, x_2, x_3$ so that:
    \begin{itemize}
        \item $Q_{1; 4:7}$, $Q_{2; 4:7}$ are linearly independent,
        \item $q_{38} \neq 0$.
    \end{itemize}
    
    Moreover, we can permute the variables $x_4, x_5, x_6, x_7$ so that one of the following conditions (denoting $d_i=q_{ii}$) is satisfied:
    \begin{enumerate}[label=(\Alph*)]
        \item $d_4, d_5 \neq 0$ and $Q_{1,2; 6,7}$ is invertible, or
        \item $d_4 \neq 0$, $C_{1,2; 6,7}$ is invertible and $Q_{1,2,8; 5,6,7}$ is invertible.
    \end{enumerate}
\end{lemma}

\begin{proof}
    First, let us prove the claim about permuting $x_1, x_2, x_3$.
    
    If we had $Q_{i; 4:7}=0$ for $1 \leqslant i \leqslant 3$, we could apply Lemma \ref{five_five_zero_lemma}, so we may assume $Q_{1; 4:7}, Q_{2; 4:7}, Q_{3; 4:7} \neq 0$. Let us consider two cases:

    $1^{\circ}$ two of the vectors $Q_{1; 4:7}, Q_{2; 4:7}, Q_{3; 4:7}$ are linearly dependent. Let $Q_{3; 4:7}=cQ_{2; 4:7}$. But then, since $\rank(C \ \gamma)=3$, we need to have $q_{28} \neq 0$ or $q_{38} \neq 0$. Swapping $x_2$ and $x_3$ if necessary, the conditions are satisfied.

    $2^{\circ}$ no two of the vectors $Q_{1; 4:7}, Q_{2; 4:7}, Q_{3; 4:7}$ are linearly independent. The thesis follows from the fact that $\rank(C)=2$, $\rank(C \ \gamma)=3$, so $\gamma \neq 0$, so $Q_{i8} \neq 0$ for some $1 \leqslant i \leqslant 3$.

    This finishes the proof of the first part of the lemma.

    Now, we prove the claim about variables $x_4, x_5, x_6, x_7$. We consider a couple of cases, depending on how many of $d_4, d_5, d_6, d_7$ are nonzero:

    $1^{\circ} \ d_4, d_5, d_6, d_7 \neq 0.$ Since (from the considerations above) $\rank(Q_{1,2; 4:7})=2$, there exist indices $i,j$ such that $Q_{1,2; i,j}$ is invertible. We permute the variables so that $Q_{1,2; 6,7}$ is invertible. Then $(A)$ is satisfied.

    $2^{\circ}$ exactly three of $d_4, d_5, d_6, d_7$ are non-zero. Let $d_4, d_5, d_6 \neq 0, \ d_7=0$.

    $2.1^{\circ} Q_{1,2; 7} \neq 0$. In this case, since $\rank(Q_{1,2; 4:7})=2$, there exists $i$ such that $Q_{1,2; i,7}$ is invertible. Swapping $x_6$ with $x_i$ we can ensure that $(A)$ is satisfied.

    $2.2^{\circ} Q_{1,2; 7}=0$. In this case we must have $q_{87}\neq0$ (as otherwise the entire $7^{\text{th}}$ column would have been $0$. Moreover, since $\rank(Q_{1,2; 4:7})=2$, there exist $i,j \neq 7$ such that $Q_{1,2; i,j}$ is invertible. Swapping the variables around so that $d_5=0$ and $Q_{1,2; 6,7}$ is invertible we can guarantee that $(B)$ is satisfied.

    $3^{\circ}$ exactly two of $d_4, d_5, d_6, d_7$ are non-zero. Let $d_4, d_5 \neq 0$, $d_6, d_7=0$. If $Q_{1,2; 6,7}$ is invertible, $(A)$ is satisfied. We also cannot have $Q_{1,2; 6,7}=0$, as then the $6^{\text{th}}$ and $7^{\text{th}}$ columns of $Q$ would have been linearly dependent. Hence, we have $\rank(Q_{1,2; 6,7})=1$ and (permuting $x_6, x_7$ if necessary) $Q_{1,2; 7} \neq 0$, $Q_{1,2; 6}=c Q_{1,2; 7}$. Since $\rank(Q_{1,2; 4:7})=2$, we have for some $i \neq 6,7$ $Q_{1,2; i,7}$ invertible. Let $i=5$. Finally, we note that we cannot have $q_{86}=cq_{87}$ (as $Q$ is invertible). Hence, we can check that $\det(Q_{1,2,8; 5,6,7}) \neq 0$. Swapping $x_5, x_6$ we can guarantee that $(B)$ is satisfied.

    $4^{\circ}$ exactly one of $d_4, d_5, d_6, d_7$ is non-zero. Let $d_4 \neq 0$, $d_5, d_6, d_7=0$. Since in this case $\rank(Q_{[8]; 5,6,7})=\rank(Q_{1,2,8; \ 5,6,7})$, $Q_{1,2,8; \ 5,6,7}$ must be invertible. Hence, for some $5 \leqslant i,j \leqslant 7$ we also have $Q_{1,2; i,j} \neq 0$ (without loss of generality $i,j =6,7$), so $(B)$ is satisfied.

    $5^{\circ}$ $d_4= d_5=d_6=d_7=0$. This is not possible, as then we would have $\rank(Q_{[8]; 4:7}) \leqslant \rank(C)+\rank(\xi) \leqslant 2+1=3$, contradicting the assumption that $Q$ is non-degenerate.
\end{proof}

\begin{proof}[Proof of Lemma \ref{case_233_with_h_0}]
    We permute the variables as in Lemma \ref{structure_lemma_233_case}.
    
    Let us denote:
    \begin{equation}
        v=2\sum_{i=1}^7 q_{8i} x_i + q_{88}x_8, \ \ u_i=\sum_{j=4}^7 q_{ij} x_j \ \text{for} \ 1 \leqslant i \leqslant 2,
    \end{equation}
    and let $a_1, a_2$ be such that $Q_{3; 4:7} = a_1Q_{1; 4:7}+a_2Q_{2; 4:7}$ for some $a_1, a_2$ (which exist by Lemma \ref{structure_lemma_233_case}).
    We need to bound by $O(X^6)L^{-K}$ the following: \begin{multline}
        \sum_{\substack{x \in [x]^8, \\ v, u_1^2 \ll X}} \Lambda(x) \int_{\alpha \in \mathfrak{m}} e \bigg( \alpha \Big( x_1u_1 + x_2u_2 + x_3(a_2u_1+a_2u_2) + (x_1^3)^T E x_1^3 + \sum_{i=4}^7 d_i x_i^2 + vx_8 - N \Big) \bigg) d\alpha \cdot \\
        \cdot \int_{\beta_1^2, \gamma \in [0,1]} e \bigg( \sum_{i=1}^2 \beta_i \Big( u_i - \sum_{j=1}^4 q_{ij} x_j \Big) + \gamma \Big( v - 2\sum_{j=1}^7 q_{8j}x_j - q_{88}x_8 \Big) \bigg) \ d\beta_1 \ d\beta_2 \ d\gamma
    \end{multline}

    Let us divide the variables into the following groups: $$A = (x_8, x_5, x_6, x_7, v), \ B = (x_1, x_2, x_3, u_1, u_2), \ C=(x_4).$$
    Now let us check the conditions of Lemma \ref{the_main_tool_lemma}.

    $(1)$: Satisfied, with $P_A=\begin{pmatrix}
        0 & 0 & 0 & 0 & 1/2\\
        0 & d_5 & 0 & 0 & 0\\
        0 & 0 & d_6 & 0 & 0\\
        0 & 0 & 0 & d_7 & 0\\
        1/2 & 0 & 0 & 0 & 0\\
    \end{pmatrix}$, $P_B=\begin{pmatrix}
        e_{11} & e_{12} & e_{13} & 1 & 0\\
        e_{21} & e_{22} & e_{23} & 0 & 1\\
        e_{31} & e_{32} & e_{33} & a_1 & a_2\\
        1 & 0 & a_1 & 0 & 0\\
        0 & 1 & a_2 & 0 & 0\\
    \end{pmatrix}$ and $c=d_4$.

    $(2)$: From Lemma \ref{structure_lemma_233_case}, we may assume $d_4 \neq 0$.

    $(3A):$ We have $\mathcal{A}=\begin{pmatrix}
        q_{88} & 2q_{85} & 2q_{86} & 2q_{87} & -1\\
        0 & q_{15} & q_{16} & q_{17} & 0\\
        0 & q_{25} & q_{26} & q_{27} & 0\\
    \end{pmatrix}$. From the discussion at the beginning $\begin{pmatrix}
        q_{16} & q_{17}\\
        q_{26} & q_{27}\\
    \end{pmatrix}$ is invertible, so so is $\mathcal{A}_2$.

    $(3B):$ We have $\mathcal{B}=\begin{pmatrix}
        2q_{81} & 2q_{82} & 2q_{83} & 0 & 0\\
        0 & 0 & 0 & -1 & 0\\
        0 & 0 & 0 & 0 & -1\\
    \end{pmatrix}$. Since we have $q_{83} \neq 0$, $\mathcal{B}_2$ is invertible.

    $(4A):$ We have \begin{multline*}
        -\mathcal{A}_2^{-1} \cdot \mathcal{A}_1=-\begin{pmatrix}
        2q_{86} & 2q_{87} & -1\\
        q_{16} & q_{17} & 0\\
        q_{26} & q_{27} & 0\\
    \end{pmatrix}^{-1} \cdot \begin{pmatrix}
                q_{88} & 2q_{85}\\
                0 & q_{15}\\
                0 & q_{25}
            \end{pmatrix} \\
    = \frac{-1}{q_{16}q_{27}-q_{17}q_{26}}\begin{pmatrix}
                0 & q_{27} & -q_{17}\\
                0 & -q_{26} & q_{16}\\
                q_{26}q_{17} - q_{16}q_{27} & 2(q_{86}q_{27} - q_{87}q_{26}) & 2(q_{16}q_{87}-q_{17}q_{86})\\
            \end{pmatrix} \cdot \begin{pmatrix}
                q_{88} & 2q_{85}\\
                0 & q_{15}\\
                0 & q_{25}
            \end{pmatrix}\\
            =\frac{1}{q_{16}q_{27}-q_{17}q_{26}}\begin{pmatrix}
                0 & -q_{15}q_{27}+q_{17}q_{25}\\
                0 & -q_{25}q_{16}+q_{26}q_{15}\\
                q_{88}(q_{16}q_{27}-q_{26}q_{17} + ) & -D
            \end{pmatrix},
    \end{multline*} where $D=2\begin{vmatrix}
        q_{15} & q_{16} & q_{17}\\
        q_{25} & q_{26} & q_{27}\\
        q_{85} & q_{86} & q_{87}\\
    \end{vmatrix}$
    
    Hence, we need to show that the rank fo the following it at least $2$: $$\frac{1}{q_{16}q_{27}-q_{17}q_{26}}\begin{pmatrix}
            0 & 0 & 0 & 0 & 1/2\\
            0 & d_{5} & 0 & 0 & 0\\
            0 & 0 & d_6 & 0 & 0\\
            0 & 0 & 0 & d_{7} & 0\\
            1/2 & 0 & 0 & 0 & 0\\
        \end{pmatrix} \cdot \begin{pmatrix}
            q_{16}q_{27}-q_{26}q_{17} & 0\\
            0 & q_{16}q_{27}-q_{26}q_{17}\\
            0 & -q_{15}q_{27}+q_{17}q_{25}\\
            0 & -q_{25}q_{16}+q_{26}q_{15}\\
            q_{88}(q_{16}q_{27}-q_{26}q_{17}) & -D
        \end{pmatrix}.$$ The first row of this product is $\begin{pmatrix} q_{16}q_{27}-q_{17}q_{26} & -D/2 \end{pmatrix}$, the second row is $\begin{pmatrix}0 & q_{5}\end{pmatrix}$ and the last row is $\begin{pmatrix} 1/2 & 0\end{pmatrix}$. Since we have either $D \neq 0$ or $d_5 \neq 0$ from the discussion at the beginning, this is satisfied.

    $(4B):$ We need the rank of $P_B \cdot \begin{pmatrix}
        \begin{matrix}
            1 & 0\\
            0 & 1
        \end{matrix}\\
        -\mathcal{B}_2^{-1} \cdot \mathcal{B}_1
    \end{pmatrix}$ to be at least $2$. We will show that in particular, the bottom two rows of this matrix are linearly independent. Hence, we need the following to have rank $2$:
    $$
        \begin{pmatrix}
        1 & 0\\
        0 & 1\\
        a_1 & a_2\\
        0 & 0\\
        0 & 0\\
    \end{pmatrix}^T \cdot \begin{pmatrix}
            \begin{matrix}
                1 & 0\\
                0 & 1
            \end{matrix}\\
            - \begin{pmatrix}
            \begin{pmatrix}
                1/{2q_{83}} & 0 & 0\\
                0 & 1 & 0\\
                0 & 0 & 1
            \end{pmatrix} \cdot
            \begin{pmatrix}
                2q_{81} & 2q_{82}\\
                0 & 0\\
                0 & 0\\
            \end{pmatrix}
            \end{pmatrix}
        \end{pmatrix} = \begin{pmatrix}
        1 & 0\\
        0 & 1\\
        a_1 & a_2\\
        0 & 0\\
        0 & 0\\
    \end{pmatrix}^T \cdot \begin{pmatrix}
            1 & 0\\
            0 & 1\\
            -q_{81}/q_{83} & -q_{82}/q_{83}\\
            0 & 0\\
            0 & 0\\
        \end{pmatrix}.
    $$ Hence, we need (since $q_{83} \neq 0$) $\begin{vmatrix}
        q_{83}-a_1q_{81} & -a_1q_{82}\\
        -a_2q_{81} & q_{83}-a_2q_{82}\\
    \end{vmatrix} = q_{83}-a_1q_{81}-a_2q_{82} \neq 0$. This is the case, as $\rank(C \ \gamma) = 3$ and $C_{3} = a_1C_1+a_2C_2$ where $C_i$ denotes the $i^{\text{th}}$ row of $C$.
\end{proof}

\section{Proof in the case when any off-diagonal, invertible submatrix of \texorpdfstring{$Q$}{Q} is of type 0 or 1 (both vertically and horizontally).}
\label{section_twos}

In this section, we prove Theorem \ref{the_main_theorem} for $Q$ not satisfying the condition of either of the two previous sections.

\begin{proposition}
\label{proposition_01}
    Let $Q$ be an invertible, symmetric matrix of off-diagonal rank $3$, encoding a quadratic form in $8$ variables.

    If any off-diagonal, invertible $3 \times 3$ submatrix of $Q$ is of type at most 1 both vertically and diagonally, then we can bound its minor arc: $$\int_{\alpha \in \mathfrak{m}} \sum_{x \in [X]^8} \Lambda(x) e \big( \alpha(Q(x) - N) \big) \ d\alpha = O(X^6L^{-K}).$$
\end{proposition}

In the previous two sections we have considered the case when $Q$ has a submatrix of type $2$ or $3$ (horizontally or vertically). Hence, in this section we will assume that all $3 \times 3$ off-diagonal, invertible submatrices of $Q$ are of type $0$ or $1$.

Since $Q$ has off-diagonal rank $3$, it has such a matrix. Without loss of generality $M=Q_{1:3; 6:8}$ is invertible. By permuting the variables $x_1, x_2, x_3$ and $x_6, x_7, x_8$, we can ensure that $\rank(H_{M, 8}) \leqslant \rank(H_{M, 7}) \leqslant \rank(H_{M, 6})$ and $\rank(V_{M, 1}) \leqslant \rank(V_{M, 2}) \leqslant \rank(V_{M, 3})$.

Let us now break the proof into a few cases.

$1^{\circ}$. If $M$ is of type $0$ both vertically and diagonally, then from Lemma \ref{three_twos_lemma} applied twice, we get that $q_{4i}=q_{i4}=0$ for any $i \neq 4$, so we may apply Lemma \ref{zero_column_lemma}. Hence, we may assume $Q_{1:3; \ 6:8}$ is of type 1 horizontally.

$2^{\circ}$. $M$ is of type $1$ in at least one of the directions (so we may assume it is of type $1$ horizontally). From Lemma \ref{two_twos_lemma}, we have $Q=\begin{pmatrix}
    - & - & - &a_1b_1 &a_1b_2 &a_1b_3 & - & - &\\
    - & - & - &a_2b_1 &a_2b_2 &a_2b_3 & - & - &\\
    - & - & - &a_3b_1 &a_3b_2 &a_3b_3 & - & - &\\
   a_1b_1 &a_2b_1 &a_3b_1 & d_1+hb_1b_1 & hb_1b_2 & hb_1b_3 & - & - &\\
   a_1b_2 &a_2b_2 &a_3b_2 & hb_2b_1 & d_2+hb_2b_2 & hb_2b_3 & - & - &\\
   a_1b_3 &a_2b_3 & a_3b_3 & hb_3b_1 & hb_3b_2 & d_3+hb_3b_3 & - & - &\\
    - & - & - & - & - & - & - & - &\\
    - & - & - & - & - & - & - & - &\\
\end{pmatrix}$. Since $M=Q_{1:3; 6:8}$ is invertible, we have $b_3 \neq 0$. Moreover, since $M$ is of type $1$ horizontally, we also have $b_1 \neq 0$ or $b_2 \neq 0$ (we may assume $b_2 \neq 0$ by swapping $x_4, x_5$ if necessary). We will now consider subcases, depending on whether $h\neq 0$ or $h=0$.

$2.1^{\circ}$. $h \neq 0$. From our assumption that $M$ is of type $0$ or $1$ vertically and that $\rank(V_{M, 1}) \leqslant \rank(V_{M, 2}) \leqslant \rank(V_{M, 3})$, applying Lemmas \ref{two_twos_lemma} or \ref{three_twos_lemma}, we get that $\rank(Q_{3:5; 6:8})=1$. Since $hb_2b_3 \neq 0$, we may write: $q_{4i}=\frac{b_1}{b_2}q_{5i}$ and $q_{3i}=\frac{a_3}{hb_2}q_{5i}$ for $i \in \{7,8\}.$ If $b_1=0$, we can apply Lemma \ref{zero_column_lemma}, so we may assume $b_1 \neq 0$.

We note that $N=Q_{1:3; 5,7,8}$ is also off-diagonal and invertible. Hence, from Lemmas \ref{two_twos_lemma} and \ref{three_twos_lemma}, we have $\rank(Q_{4,6;5,7,8}) \leqslant 1$. Therefore, $q_{4i}=\frac{b_1}{b_3}q_{6i}$ for $i \in \{7, 8\}$. Hence, as $b_1\neq 0$, we have $q_{6i}=\frac{b_3}{b_2}q_{5i}$ for $i \in \{7,8\}$. Hence, we're in a situation from Lemma \ref{three_columns_rank_1} (the fourth, fifth and sixth rows of $Q$ have rank $1$ after adjusting their diagonal entries appropriately). This finishes the proof in this case.

$2.2^{\circ}$ $h=0$. Let us split this case into two cases, depending on the type of $M$ vertically.

$2.2.1^{\circ}$ $M$ is also of type $1$ vertically. Let $h'$ denote the vertical counterpart of parameter $h$ from Lemma \ref{two_twos_lemma} for $M$. If $h' \neq 0$, we can apply the Case $2.1^{\circ}$, so we may assume $h'=0$ as well. Hence, after recalling that $\rank(H_{M, 8}) \leqslant \rank(H_{M, 7}) \leqslant \rank(H_{M, 6})$ and $\rank(V_{M, 1}) \leqslant \rank(V_{M, 2}) \leqslant \rank(V_{M, 3})$ and applying Lemma \ref{two_twos_lemma} twice, we get that $Q$ is of the form:
\begin{equation}
    Q=\begin{pmatrix}
    - & - & - &a_1b_1 &a_1b_2 &a_1b_3 & - & - &\\
    - & - & - &a_2b_1 &a_2b_2 &a_2b_3 & - & - &\\
    - & - & d_1 &a_3b_1=0 &a_3b_2=0 &a_3b_3=u_3v_3 & u_3v_2 & u_3v_1 &\\
   a_1b_1 &a_2b_1 &a_3b_1=0 & d_2 & 0 & 0=u_2v_3 & u_2v_2 & u_2v_1 &\\
   a_1b_2 &a_2b_2 &a_3b_2=0 & 0 & d_3 & 0=u_1v_3 & u_1v_2 & u_1v_1 &\\
   a_1b_3 &a_2b_3 & a_3b_3=u_3v_3 & 0=u_2v_3 & 0=u_1v_3 & d_4 & - & - &\\
    - & - & u_3v_2 & u_2v_2 & u_1v_2 & - & - & - &\\
    - & - & u_3v_1 & u_2v_1 & u_1v_1 & - & - & - &\\
\end{pmatrix}.
\end{equation}
Since $b_2 \neq 0$ and $b_2a_3=0$, we have $a_3=0$ and so $q_{36}=q_{63}=0$.

Since $M$ is invertible, we have $u_3 \neq 0$. Moreover, since it is of type 1 vertically, at least one of $u_2, u_1$ is nonzero.
We will now show that in either case, $\rank(Q_{1,2; 3:6})=1$ and $\rank(Q_{3:6;7,8})=1$.

If $u_2 \neq 0$, let us consider $Q_{1:3; 5,7,8}$, which is also off-diagonal and invertible. Hence, it is vertically of type 0 or 1, so $\rank(Q_{4,6; 5,7,8}) \leqslant 1$. Since $u_2 \neq 0$, we may write $q_{67}=\frac{u_4}{u_2}v_2$ and $q_{68}=\frac{u_4}{u_2}v_1$ for some $u_4$. Hence, $\rank(Q_{3:6;7,8})=1$. We get that $\rank(Q_{1,2; 3:6})=1$ by considering $Q_{1,2,4; 6:8}$ in an analogous way.

Now, let us consider the case when $u_2=0$ (and so $u_1 \neq 0$). If also $b_1=0$, then we may apply Lemma \ref{zero_column_lemma} for variable $x_4$, so we may assume $b_1 \neq 0$. However, then we may apply exactly the same reasoning as above, for $Q_{1:3; 4,7,8}$ and $Q_{1,2,5; 6:8}$.

Hence, we've got that
$Q=\begin{pmatrix}
    E & a b^T & F\\
    ba^T & D & l k^T\\
    F^T & k l^T & G
\end{pmatrix},$ for $E, F, G \in M_{2 \times 2}(\z)$ ($E, G$ symmetric), $D \in M_{4 \times 4}(\z)$ diagonal, $a, k \in M_{1 \times 2}(\q)$ and $b, l \in M_{1 \times 4}(\q)$. We will now use Lemma \ref{the_main_tool_lemma} to finish the solution in this case. We note that here we redefined $a, b$.

Let us denote: $z_1, z_2=x_1, x_2$, $y_1, y_2, y_3, y_4=x_3, x_4, x_5, x_6$, $t_1, t_2=x_7, x_8$,

We note that if $F=0$, we could finish the proof by applying Lemma \ref{the_main_tool_lemma} with two extra variables: $w=\sum_{i=1}^4b_iy_{i}$ and $s=\sum_{i=1}^4l_iy_{i}$. Since this case is simpler than the case when $F \neq 0$, I will not elaborate on it and instead present the solution in the case when $F \neq 0$ (after swapping variables $x_1, x_2$ and $x_7, x_8$ if necessary, we may assume $f_{21}=q_{27}\neq 0$).

We also note that if $a_1, a_2, k_1$ or $k_2$ is $0$, then we may apply Lemma \ref{five_five_zero_lemma}, so we may assume these are nonzero.

Now we will show that we may permute variables $y_3, y_4, y_5, y_6$ (i.e., $x_3, x_4, x_5, x_6$) so that $d_3, d_4 \neq 0$ and $\begin{pmatrix}
    b_1 & b_2\\
    l_1 & l_2
\end{pmatrix}$ is invertible. First, we note that if $\rank(b, l) \leqslant 1$, we may apply Lemma \ref{three_columns_rank_1}, so we may assume $\rank(b, l)=2$. We also note that since $Q$ is invertible, at least two of the diagonal entries of $D$ must be nonzero. Let us not consider three cases:
$1^{\circ}.$ $d_1, d_2, d_3, d_4 \neq 0$. This case is trivial - if for some $i, j$ $\begin{pmatrix}
    b_i & b_j\\
    l_i & l_j
\end{pmatrix}$ is invertible, we can just swap $y_i, y_j$ with $y_1, y_2$ to obtain the desired configuration.
$2^{\circ}.$ Exactly one of $d_1, d_2, d_3, d_4$ is $0$ (let us say that $d_i=0$). Since $Q$ is invertible, we must have $\begin{pmatrix}
    b_i \\ l_i
\end{pmatrix}\neq 0$. Hence, since $\rank(b, l)=2$, $\begin{pmatrix}
    b_i & b_j \\ l_i & l_j
\end{pmatrix}$ is invertible for some $j$. Swapping $y_i, y_j$ with $y_1,y_2$ we obtain the required permutation.
$3^{\circ}.$ Exactly two of $d_1, d_2, d_3, d_4$ are nonzero (let us say that $d_i=d_j$). Since $Q$ is invertible, $\begin{pmatrix}
    b_i & b_j\\
    l_i & l_j
\end{pmatrix}$ must be invertible as well, so we again can just swap $y_i, y_j$ with $y_1,y_2$, obtaining the desired configuration.

Having made these initial observations on the properties of $Q$, we will now apply Lemma \ref{the_main_tool_lemma} to finish the proof in this case.

Let us denote:
\begin{equation}
    w=\sum_{i=1}^4b_iy_i, \ \ s_1=2\sum_{i=1}^6 q_{7i}x_i+q_{77}t_1, \ \ s_2=2\sum_{i=1}^7 q_{8i}x_i+q_{88}t_2.
\end{equation}
We need to bound by $O(X^6L^{-K})$ the following: \begin{multline}
    \sum_{\substack{z, t \in [X]^2, \ y \in [X]^4, \\ w, s_1, s_2 \ll X}} \Lambda(z)\Lambda(t)\Lambda(y) \int_{\alpha \in \mathfrak{m}} e \bigg(  \alpha \Big( z^TEz+2w(a_1z_1+a_2z_2)+\sum_{i=1}^4d_iy_i^2+s_1t_1+s_2t_2 - N \Big) \bigg) \ d\alpha \ \cdot \\
    \cdot \int_{\beta, \gamma_1^2 \in [0, 1]} e \bigg( \beta \Big( u - \sum_{i=1}^4 b_iy_i \Big) + \gamma_1 \Big( s_1 - 2\sum_{i=1}^6 q_{7i}x_i-q_{77}t_1 \Big) + \gamma_2 \Big( s_2 - 2\sum_{i=1}^7 q_{8i}x_i-q_{88}t_2 \Big) \bigg) \ d\beta \ d\gamma_1 \ d\gamma_2.
\end{multline}

Let us divide the variables into groups as follows: $$A=(t_1, y_3, y_1, y_2, s_1), \ B=(t_2, x_1, x_2, s_2, u), \ C=(y_4).$$ Now let us check that the conditions of Lemma \ref{the_main_tool_lemma} are satisfied.

$(1):$ This condition is satisfied, for:
\begin{equation*}
    P_A=\begin{pmatrix}
    0 & 0 & 0 & 0 & 1/2\\
    0 & d_3 & 0 & 0 & 0\\
    0 & 0 & d_1 & 0 & 0\\
    0 & 0 & 0 & d_2 & 0\\
    1/2 & 0 & 0 & 0 & 0\\
\end{pmatrix},
\end{equation*}
\begin{equation*}
    P_B=\begin{pmatrix}
    0 & 0 & 0 & 1/2 & 0\\
    0 & e_{11} & e_{12} & 0 & a_1\\
    0 & e_{21} & e_{22} & 0 & a_2\\
    1/2 & 0 & 0 & 0 & 0\\
    0 & a_1 & a_2 & 0 & 0\\
\end{pmatrix}
\end{equation*}
and $c=d_4$.

$(2):$ This condition is satisfied, as $d_4 \neq 0$.

$(3A):$ We have $\mathcal{A}=\begin{pmatrix}
    0 & b_3 & b_1 & b_2 & 0\\
    q_{77} & 2k_1l_3 & 2k_1l_1 & 2k_1l_2 & -1\\
    2q_{78} & 2k_2l_3 & 2k_2l_1 & 2k_2l_2 & 0
\end{pmatrix}$, so $\mathcal{A}_2=\begin{pmatrix}
    b_1 & b_2 & 0\\
    2k_1l_1 & 2k_1l_2 & -1\\
    2k_2l_1 & 2k_2l_2 & 0\\
\end{pmatrix}$ is invertible as so is $\begin{pmatrix}
    b_1 & b_2\\
    l_1 & l_2
\end{pmatrix}$ (and $k_2 \neq 0$).

$(3B):$ We have $\mathcal{B}=\begin{pmatrix}
    0 & 0 & 0 & 0 & -1\\
    0 & 2q_{17} & 2q_{27} & 0 & 0\\
    q_{88} & 2q_{18} & 2q_{28} & -1 & 0\\
\end{pmatrix}$, so $\mathcal{B}_2=\begin{pmatrix}
    0 & 0 & -1\\
    2q_{27} & 0 & 0\\
    2q_{28} & -1 & 0\\
\end{pmatrix}$ is invertible as we assumed that $q_{27} \neq 0$.

$(4A):$
We need to show that $\begin{pmatrix}
    0 & 0 & 0 & 0 & 1/2\\
    0 & d_3 & 0 & 0 & 0\\
    0 & 0 & d_1 & 0 & 0\\
    0 & 0 & 0 & d_2 & 0\\
    1/2 & 0 & 0 & 0 & 0\\
\end{pmatrix} \cdot \begin{pmatrix}
    \begin{matrix}
        1 & 0\\
        0 & 1\\
    \end{matrix}\\
    -\mathcal{A}_2^{-1}\mathcal{A}_1
\end{pmatrix}$ has rank at least $2$, which is clear, as the second and fifth rows of the product are, respectively, $\begin{pmatrix}0 & d_3\end{pmatrix}$ and $\begin{pmatrix}1/2 & 0\end{pmatrix}$ and $d_3 \neq 0$.

$(4B):$ We need to show that
$\begin{pmatrix}
    0 & 0 & 0 & 1/2 & 0\\
    0 & e_{11} & e_{12} & 0 & a_1\\
    0 & e_{21} & e_{22} & 0 & a_2\\
    1/2 & 0 & 0 & 0 & 0\\
    0 & a_1 & a_2 & 0 & 0\\
\end{pmatrix} \cdot \begin{pmatrix}
    \begin{matrix}
        1 & 0\\
        0 & 1
    \end{matrix}\\
    -\mathcal{B}_2^{-1}\cdot \mathcal{B}_1
\end{pmatrix}$ has rank at least $2$. Its fourth row is $\begin{pmatrix}1/2 & 0\end{pmatrix}$, so it suffices to show that it has a nonzero entry in its second column.
We have:
$$-\mathcal{B}_2^{-1}\cdot \mathcal{B}_1=-\begin{pmatrix}
    0 & \frac{1}{2q_{27}} & 0\\
    0 & \frac{q_{28}}{q_{27}} & -1\\
    -1 & 0 & 0\\
\end{pmatrix} \cdot \begin{pmatrix}
    0 & 0\\
    0 & 2q_{17}\\
    q_{88} & 2q_{18}\\
\end{pmatrix}=\begin{pmatrix}
    0 & -\frac{q_{17}}{q_{27}} \\
    q_{88} & 2\frac{q_{27}q_{18}-q_{17}q_{28}}{q_{27}} \\
    0 & 0 \\
\end{pmatrix},$$ so the second column of the product is: $$\begin{pmatrix}
    0 & 0 & 0 & 1/2 & 0\\
    0 & e_{11} & e_{12} & 0 & a_1\\
    0 & e_{21} & e_{22} & 0 & a_2\\
    1/2 & 0 & 0 & 0 & 0\\
    0 & a_1 & a_2 & 0 & 0\\
\end{pmatrix} \cdot \begin{pmatrix}
    0\\ 1\\ 
    -\frac{q_{17}}{q_{27}} \\
    2\frac{q_{27}q_{18}-q_{17}q_{28}}{q_{27}} \\
    0 \\
\end{pmatrix}=\begin{pmatrix}
    \frac{q_{27}q_{18}-q_{17}q_{28}}{q_{27}}\\
    e_{11}-\frac{q_{17}}{q_{27}}e_{12}\\
    e_{21}-\frac{q_{17}}{q_{27}}e_{22}\\
    0\\
    a_1-a_2\frac{q_{17}}{q_{27}}
\end{pmatrix}.$$ We notice that if all the entries in this column are $0$, then $q_{27} \cdot Q_{1; [8]} - q_{17} \cdot Q_{2; [8]}=0$, contradicting the assumption that $Q$ is invertible. Hence, the condition $(4B)$ is satisfied and we may apply Lemma \ref{the_main_tool_lemma} to finish the proof in this case.

$2.2.2^{\circ}$ $M$ is of type $0$ vertically. From Lemmas \ref{two_twos_lemma} and \ref{three_twos_lemma}, we have:
\begin{equation}
    Q=\begin{pmatrix}
    - & - & - &a_1b_1 &a_1b_2 &a_1b_3 & - & - &\\
    - & - & - &a_2b_1 &a_2b_2 &a_2b_3 & - & - &\\
    - & - & - &a_3b_1 &a_3b_2 &a_3b_3 & - & - &\\
   a_1b_1 &a_2b_1 &a_3b_1 & d_1 & 0 & 0 & 0 & 0 &\\
   a_1b_2 &a_2b_2 &a_3b_2 & 0 & d_2 & 0 & 0 & 0 &\\
   a_1b_3 &a_2b_3 & a_3b_3 & 0 & 0 & d_3 & - & - &\\
    - & - & - & 0 & 0 & - & - & - &\\
    - & - & - & 0 & 0 & - & - & - &\\
\end{pmatrix}.
\end{equation}
Since $b_2 \neq 0$, $N=Q_{1:3; 5,6,8}$ is also an off-diagonal, invertible matrix. Horizontally it is of type $1$ (as $M$ is). Hence, if its $h'$ (the counterpart of the parameter $h$ from Lemma \ref{two_twos_lemma} for $N$) was non-zero, we could apply the case $2.1$. Hence, we may assume $h' = 0$. If $N$ was vertically of type $1$, we could apply case $2.2.1$, so let us assume that it is vertically of type $0$. Therefore, we have $q_{67}=q_{68}=0$. Hence, $Q$ satisfies the conditions of Lemma \ref{three_columns_rank_1} (its fourth, fifth and sixth rows are linearly dependent, after adjusting the diagonal entries). This finishes the proof of this case and also of the Proposition.

\section{The proof of Theorem \ref{the_main_theorem}.}

\begin{proof}[Proof of Theorem \ref{the_main_theorem}]
    Let $Q$ be a symmetric, invertible matrix with integer coefficients of off-diagonal rank $3$, representing a homogeneous quadratic form.

    Since $Q$ is of off-diagonal rank $3$, it must have an off-diagonal, invertible $3 \times 3$ submatrix $M$, which both vertically and horizontally has types 0, 1, 2 or 3.

    If $Q$ has such a matrix with (vertical or horizontal) type 3 or 2, we may apply respectively Proposition \ref{proposition_3} or \ref{proposition_2} to bound the minor arcs: $$\int_{\alpha \in \mathfrak{m}} \sum_{x \in [X]^8} \Lambda(x) e \big( \alpha(Q(x) - N) \big) \ d\alpha = O(X^6L^{-K}).$$
    If, on the contrary, all off-diagonal, invertible $3 \times 3$ submatrices of $Q$ have rank at most 1 both vertically and diagonally, we may apply Proposition \ref{proposition_01} to also bound the minor arcs by $O(X^6L^{-K})$.

    Applying Proposition \ref{zhao_major} to estimate: $$\int_{\alpha \in \mathfrak{M}} \sum_{x \in [X]^8} \Lambda(x) e \big( \alpha(Q(x) - N) \big) \ d\alpha = \mathfrak{S}_Q(N)\mathfrak{J}_Q(N, X)+ O(X^6L^{-K}),$$ we get that: $$\sum_{x \in [X]^8} \Lambda(x) \cdot 1_{Q(x)=N}=\int_{\alpha \in [0, 1]} \sum_{x \in [X]^8} \Lambda(x) e \big( \alpha(Q(x) - N) \big) \ d\alpha =\mathfrak{S}_Q(N)\mathfrak{J}_Q(N, X)+ O(X^6L^{-K}),$$ as desired.
\end{proof}

\appendix

\section{A lemma on diagonal submatrices.}

\begin{lemma}
\label{diagonal_submatrix_of_rank_lemma}
    Suppose that $Q$ is an invertible, $m \times m$ symmetric matrix. Then for any $1 \leqslant k \leqslant m$ there exists a set of $k$ indices $I=\{i_1, \dots, i_k\}$ such that $\rank(Q_{I; \ I}) \geqslant k-1$.
\end{lemma}

\begin{proof}
    We will proceed by induction on $k$. For $k=1$ the lemma is trivial.

    Now let us assume that $k \geqslant 2$ and that the thesis holds for $k-1$, but does not hold for $k$. Let us take a subset of indices $I$, $|I|=k-1$ with the largest possible rank of $Q_{I; \ I}$ (so from the inductive hypothesis we know that $\rank(Q_{I; \ I}) \geqslant k-2$). If $\rank(Q_{I; \ I}) = k-1$, then the thesis holds, so we may assume $\rank(Q_{I; \ I}) = k-2$. Let us assume that for any $i \notin I$ we have $\rank(Q_{I \cup \{i\}; \ I \cup \{i\}})=k-2$. In particular, for any $i \notin I$ we have $\rank(Q_{I; \ I \cup \{i\}})=k-2$, so (denoting $v_i=Q_{I; i}$) for any $i \notin I$ we have $v_i \in \langle v_j : j \in I \rangle$. Hence, $\rank(Q_{I; [m]})=k-2$, so there are $k-1$ row vectors of $Q$ which span a space of dimension $k-2$, while $Q$ is invertible. This contradiction finishes the proof.
\end{proof}

\end{document}